\documentclass[11pt,oneside]{amsart}
\usepackage{geometry, bm, color,enumitem,mathrsfs}
\usepackage{comment}
\usepackage{physics}
\usepackage{multicol}
\newtheorem{theorem}{Theorem}[section]
\newtheorem{corollary}[theorem]{Corollary}
\newtheorem{lemma}[theorem]{Lemma}
\newtheorem{conjecture}[theorem]{Conjecture}
\newtheorem{remark}[theorem]{Remark}
\newtheorem{example}[theorem]{Example}

\newcommand{\defn}[1]{\textit{#1}}
\newcommand*{\rttensor}[1]{\overline{\overline{#1}}}

\usepackage{stackengine,amssymb,graphicx}

\newtheorem{thm}{Theorem}

      \makeatletter
      \def\@setcopyright{}
      \def\serieslogo@{}
      \makeatother      

\begin{document}
   \author{Amin  Bahmanian}
   \address{Department of Mathematics,
  Illinois State University, Normal, IL USA 61790-4520}
   \author{Sho  Suda}
   \address{Department of Mathematics, National Defense Academy of Japan, 239-8686 Japan}

\title[Hadamard Hypercubes]{Hadamard Hypercubes}

   \begin{abstract}  
Although Hadamard matrices have been investigated since the nineteenth century, relatively little is known about their higher-dimensional analogues.  In this paper, we introduce two constructions of Hadamard hypercubes. The first construction is derived from conference matrices, while the second is recursive, combining Hadamard matrices (and hypercubes) of smaller order with Latin hypercubes. The former approach draws on the theory of association schemes on triples, whereas the latter yields applications to the construction of higher-dimensional symmetric designs.   
   \end{abstract}


   \date{\today}

   \maketitle

\section{Introduction}
Throughout this paper, $n,d,k,\ell\in \mathbb N$ and $d\geq 2$. We are primarily concerned with {\it $\pm 1$-arrays} which are arrays whose entries are either +1 or -1 (abbreviated to $+$ and $-$).   
{\it Hadamard matrices} are 2-dimensional square $\pm 1$-arrays with every two rows having matching entries in exactly half of their columns. For example, $\begin{tabular}{ c c }
 + & +  \\ 
 + & \textminus     
\end{tabular}$ is a Hadamard matrix of order two.  In connection with  a problem on tessellations, Sylvester considered Hadamard matrices in 1867  \cite{Sylv1867}.  Hadamard showed that Sylvester's construction produces examples of matrices of maximum determinant among all  $\pm 1$-matrices whose order is a power of 2 \cite{Hadam1893}. Two remarkable constructions of Hadamard matrices are due to Paley  \cite{Paley1933} whose construction is based on the quadratic character of finite fields motivated by theory of polytopes, and Williamson's method \cite{MR9590}. 
Aside from theoretical aspects, Hadamard matrices have been successfully applied in the theory of error-correcting codes \cite[Chapter 18]{MR1871828}, signal processing \cite{SignHadApp}, spectrometry \cite{HadamSpect}, and statistics \cite{MR523759}. Despite considerable efforts, the following conjecture, known as Hadamard's conjecture, which is attributed to Paley, is wide open, and it is not even known whether or not a Hadamard matrix of order 668 exists. 
\begin{conjecture}
A Hadamard matrix of order $n$  exists for every $n\equiv 0 \pmod 4$. 
\end{conjecture}

To see why divisibility by four is necessary, consider a Hadamard matrix of order $n$ where $n\geq 3$. By normalizing the matrix and permuting the columns, we can ensure that the first three rows are as follows.
\begin{align*}
\begin{array}{lccccc}
&+ \cdots + & + \cdots + & + \cdots + & + \cdots + &\\
&+ \cdots + & + \cdots + & - \cdots - & - \cdots - &\\
&
\underbrace{+ \cdots +}_{c_1 \text{ columns}} &
\underbrace{- \cdots -}_{c_2 \text{ columns}} &
\underbrace{+ \cdots +}_{c_3 \text{ columns}} &
\underbrace{- \cdots -}_{c_4 \text{ columns}}&
\end{array}
\end{align*}
Since these three rows are mutually orthogonal, we have $c_3+c_4=c_2+c_4=c_2+c_3=n/2$. This together with $c_1+c_2+c_3+c_4=n$ implies that $c_1=c_2=c_3=c_4=n/4$. 

While Hadamard's conjecture may seem discouraging, it is still natural to explore higher-dimensional analogues of Hadamard matrices. This line of research was initiated in the late 1970s by Shlichta \cite{MR545012}, Hammer and Seberry \cite{MR593700} and Seberry \cite{MR611197}. To this end, we first establish some notation.

Let $[n]=\{1,\dots,n\}$, and let $[n]^d=[n]\times \dots\times [n]$ ($d$ times).  A {\it hypercube} $H$ of order $n$ and dimension $d$ is a $d$-dimensional array whose $n^d$ cells  are indexed by $[n]^d$.  For ${\vb i}:=(i_1,\dots,i_d)\in [n]^d$,  the entry in position $(i_1,\dots,i_d)$ of $H$ is denoted by $H({\vb i})$,  and the {\it dot product} of two $d$-dimensional hypercubes $A,B$ of order $n$ is    $A\cdot B =\sum_{{\vb i}\in[n]^d} A(\vb i)B(\vb i)$.
For $\ell\in [d]$, an {\it $\ell$-layer} in $H$ is a sub-hypercube of $n^\ell$ cells obtained by fixing $d-\ell$ coordinates (the remaining $\ell$ coordinates vary). A {\it line} is a $1$-layer and a {\it hyperplane} is a $(d-1)$-layer. Two distinct layers are {\it parallel} if both are obtained by fixing the same coordinates, and are {\it orthogonal} if their dot product is zero.

For $\ell\in [d-1]$, an {\it $(n,d,\ell)$-Hadamard hypercube} is a  $\pm 1$-hypercube of order $n$ and dimension $d$ in which every two parallel $\ell$-layers are orthogonal. We refer to $\ell$ as the \emph{class} of the Hadamard hypercube. Hadamard matrices and cubes correspond to the cases $d=2,3$, respectively. 
 We remark that every  $(n,d,\ell)$-Hadamard hypercube is also an $(n,d,\ell')$-Hadamard hypercube for  $\ell'\geq \ell$.  
For example,  stacking 
\begin{tabular}{ c c }
 + & +  \\ 
 + & \textminus     
\end{tabular}
on top of  
\begin{tabular}{ c c }
+ & \textminus  \\ 
 \textminus & \textminus     
\end{tabular}
forms a $(2,3,1)$-Hadamard cube.  Yang \cite{Yangpaper86Had} observed that the existence of  $(n,d,1)$-Hadamard hypercubes is equivalent to the existence of $n\times n$ Hadamard matrices, for if $A$ is a Hadamard matrix  of order $n$, then the $d$-dimensional hypercube $H$ of order $n$ with $H(i_1,\dots,i_d)=\prod_{1\leq j<k\leq d} A(i_j,i_k)$ is an $(n,d,1)$-Hadamard hypercube. 

A necessary condition for the existence of $(n,d,d-1)$-Hadamard hypercubes is that $n$ is even (see \cite[Theorem~3.2.1]{YangNiuXu2010_HigherDimHadamard} for $d=3$, but its proof is valid for any $d$). 
One of the key open problems concerning higher-dimensional Hadamard hypercubes is how to construct those whose order 
 $n$ satisfies $n\equiv 2\pmod{4}$. 
A construction of $(2\cdot 3^n,3,2)$-Hadamard cubes is given in \cite[Theorem~3.2.3]{YangNiuXu2010_HigherDimHadamard} for any $n$. 
Recently, Kr\v cadinac, Pav\v cevi\'c, and Tabak  \cite{MR4649940}  extended  Paley's construction of Hadamard matrices to the three-dimensional case as follows. 
\begin{theorem}{\rm \cite[Theorem~2.3]{MR4649940}}\label{thm:paleycube}
    Let $\mathbb{F}_q$ be the finite field of order $q$, $q$ odd, and $\chi$ be the quadratic character on $\mathbb{F}_q$. 
Then $H:\left(\mathbb{F}_q\cup\{\infty\}\right)^3\rightarrow \{\pm 1\}$ defined as follows is a $(q+1,3,2)$-Hadamard cube.
\begin{align*}
    H(x,y,z)=\begin{cases}
    -1 & \text{ if }|\{x,y,z\}|=1,\\
    1 & \text{ if } |\{x,y,z\}|=2,\\
    \chi(z-y) &\text{ if }|\{x,y,z\}|=3,x=\infty,\\
    \chi(x-z) &\text{ if }|\{x,y,z\}|=3,y=\infty,\\
    \chi(y-x) &\text{ if }|\{x,y,z\}|=3, z=\infty,\\
    \chi\big((x-y)(y-z)(z-x)\big) &\text{ otherwise}.
    \end{cases}
\end{align*}
\end{theorem}
In Theorem~\ref{thm:paleycube}, if $q$ is a prime power and $q\equiv1 \pmod{4}$, the order of the resulting Hadamard cube is $q+1\equiv2 \pmod{4}$.

A \defn{conference matrix} is a square matrix with zero diagonal, $\pm 1$ off-diagonal entries, and pairwise orthogonal rows. We generalize  Theorem~\ref{thm:paleycube} as follows.
\begin{thm}\label{thm:main1}
Let $C$ be a symmetric or skew-symmetric conference matrix of order $n$. 
Define $H:[n]^3\rightarrow \{\pm 1\}$ by 
\begin{align*}
    H(x,y,z)=\begin{cases}
    a_0 & \text{ if }x=y=z,\\
    a_1 & \text{ if }y=z\neq x,\\
    a_2 & \text{ if }x=z\neq y,\\
    a_3 & \text{ if }x=y\neq z,\\
    a_4 & \text{ if }C(x,y) C(y,z) C(z,x)=-1,\\
    a_5 & \text{ if }C(x,y) C(y,z) C(z,x)=1, 
    \end{cases}
\end{align*}
 where $(a_i)_{i=0}^{5}\in\{1,-1\}^6$. The hypercube $H$ is an $(n,3,2)$-Hadamard cube if and only if $$a_0a_1a_2a_3=a_4a_5=-1.$$  
\end{thm}
Because conference matrices are well-studied, Theorem \ref{thm:main1} enables the construction of Hadamard cubes of class 2 for a broad spectrum of orders (see Subsection \ref{confmatlitsubsec}). In particular, Theorem \ref{thm:main1} resolves many of the previously open cases listed in \cite{YangNiuXu2010_HigherDimHadamard} and \cite[Problem 2]{MR4885925}.

To prove Theorem~\ref{thm:main1}, we use association schemes on triples (ASTs), which are three-dimensional generalizations of association schemes \cite{MR1075708}. Whereas association schemes provide combinatorial axiomatizations of transitive permutation groups, ASTs serve the same role for doubly transitive permutation groups. Examples of ASTs include those derived from group actions, Steiner triple systems, and regular two-graphs. Goethals and Seidel studied the relationship between Hadamard matrices and strongly regular graphs, showing that symmetric regular Hadamard matrices with constant diagonal entries can be expressed as a linear combination of the identity matrix and the adjacency matrices of a strongly regular graph (with suitable parameters) and its complement \cite[Theorem~4.2]{MR0282872}. For further research on (complex) Hadamard matrices derived from association schemes, we refer the reader to \cite{Chan2011ComplexSRG,Chan2020ComplexHadamard,FenderKharaghaniSuda2018,GOLDBACH1998943,IkutaMunemasa2015BoseMesner,IkutaMunemasa2018BUTSON,IkutaMunemasa20193Class,KHARAGHANI201672}. 

To construct Hadamard cubes, we study a three-dimensional analogue of the relationship between association schemes and Hadamard matrices. We begin by extending several connections between symmetric ASTs and regular two-graphs to the skew-symmetric setting, namely to skew-symmetric ASTs and regular skew two-graphs.
Note that any symmetric or skew-symmetric conference matrix gives rise to a symmetric or skew-symmetric AST, respectively. 
We then determine which linear combinations of the adjacency matrices of these ASTs produce Hadamard cubes.
We remark that the Hadamard cubes constructed in Theorem~\ref{thm:main1} based on conference matrices obtained from the quadratic characters with the coefficients $(a_i)_{i=0}^5=(-1,1,1,1,1,-1)$ encompass those constructed in Theorem \ref{thm:paleycube}. 
Thus, Theorem~\ref{thm:main1} refines Theorem~\ref{thm:paleycube} in two ways: it extends the construction from Paley-type cases to general conference matrices, and it proceeds via association schemes on triples, from which the Hadamard cubes are then derived.

In another direction, to construct Hadamard cubes, we generalize Kharaghani's construction of Hadamard matrices using Latin squares \cite{MR810264}. To do so,
we introduce various constructions of Latin hypercubes. We obtain the following recursive construction for Hadamard hypercubes.
\begin{thm}\label{introthm:hypercube}
Let $n, r, d \in \mathbb{N}$ with $d \geq 2$. If there exists a Hadamard matrix of order $n$, then there exists an $(n^2,d,1)$-Hadamard hypercube. If there exists an $(n^r, s+1, s)$-Hadamard hypercube, then there also exists an $(n^{r+1}, ds, ds-1)$-Hadamard hypercube.
\end{thm}

In particular, when $s = 1$, Theorem~\ref{introthm:hypercube} implies that the existence of a Hadamard matrix of order $n^r$ guarantees the existence of an $(n^{r+1}, d, d-1)$-Hadamard hypercube for any $d\geq 2$. Moreover, by combining Theorems~\ref{thm:main1} and \ref{introthm:hypercube}, one obtains Hadamard hypercubes spanning a broad range of orders, dimensions, and classes.

An incidence matrix of a {\it symmetric $(n,k,\lambda)$ design} is an $n\times n$ $\{0,1\}$-matrix in which every row and column has exactly $k$ ones, and the dot product of any two distinct rows or columns is~$\lambda$. Constructing an infinite family of symmetric designs or even a single symmetric design is a notoriously difficult problem \cite{MR2234039}. A \defn{symmetric $(n,d, k,\lambda)$-design} is a  $\{0,1\}$-hypercube of order $n$ and dimension $d$ with the property that every 
 $2$-layer is the incidence matrix of a symmetric $(n,k,\lambda)$-design. 
In \cite{MR1126968}, de Launey investigated several classes of higher-dimensional combinatorial structures. Our next result, complements  recent developments in this area  due to   {Kr{\v c}adinac, Pav{\v c}evi\'c, Osvin and Tabak \cite{MR4870040,Krčadinac01072025}.

\begin{thm}\label{introthm:hcubeproper}

    If there exists a Hadamard matrix of order $n$, then there exists a symmetric $(n^2,d, n(n-1)/2,n(n-2)/4)$-design for any $d\geq 2$. 
\end{thm}

The organization of this paper is as follows. 
In Section~\ref{sec:prelim}, we review the necessary background on conference matrices, association schemes on triples, and the correspondence between regular two-graphs and certain  association schemes on triples. In Section~\ref{sec:skewtwograph}, we present an analogue of the result from Subsection~\ref{sec:astrtt} for regular skew two-graphs, which is of independent interest in the study of association schemes on triples. Building on the association schemes on triples developed in Sections~\ref{sec:prelim} and~\ref{sec:skewtwograph}, we show in Section~\ref{sec:hcast} that an appropriate linear combination of the adjacency matrices in the AST, with coefficients $1$ and $-1$, yields a Hadamard cube. In Section~\ref{sec:revursive}, we provide a recursive construction of higher-dimensional Hadamard matrices by making use of line-rainbow  Latin hypercubes. We prove theorem \ref{introthm:hcubeproper} in Section \ref{subsec:hcube1}.

\section{Preliminaries}\label{sec:prelim}
For a finite set $X$, $\binom{X}{3}$ is the collection of all 3-subsets of $X$, and $S_X$ is  the symmetric group on  $X$, and we write $S_n$ when $X=\{1,\dots,n\}$. For a real number $r$, a matrix is $r$-diagonal, if all its diagonal entries are $r$. 

Let $J$ denote  the  $n\times n\times n$ matrix where every entry is equal to one.
For three $n\times n\times n$ matrices $A$, $B$,  $C$, the \emph{ternary product}
$D=ABC$ is the $n\times n\times n$ matrix with entries
\begin{align} \label{ternaryprod}
D(x,y,z)=\sum_{w\in X}A(w,y,z)B(x,w,z)C(x,y,w)\quad \mbox{ for } x,y,z\in X.    
\end{align}

Throughout this paper, $X:=\{1,\dots,n\}$, and $X^k:=X\times \dots\times X$ ($k$ times). For a permutation $\sigma\in S_k$, and ${\vb x}:=(x_1,\dots,x_k)\in R\subseteq X^k$, let $\sigma(\vb x)=(x_{\sigma(1)},\dots,x_{\sigma(k)})$, and 
$\sigma(R)=\{\sigma(\vb y) \mid \vb y\in R\}$. We say that the relation $R$ is \emph{symmetric with respect to $\sigma$} if $\sigma(R)=R$, and  \emph{symmetric} if this condition holds for  every permutation $\sigma$ of $\{1,\dots,k\}$.  For a permutation $\sigma\in S_k$ and a $k$-dimensional hypercube $A$, let $A^ \sigma$ 
 be the hypercube with $A^ \sigma(\vb x)=A(\sigma(\vb x))$ for ${\vb x}\in X^k$.

\subsection{Conference matrices} \label{confmatlitsubsec}
A \defn{signed permutation matrix} is a square matrix over $\{0,\pm 1\}$ such that each row and column has exactly one non-zero entry. A square matrix $A$ is \defn{symmetric} (\defn{skew-symmetric}) if $A^\top =A$ ($A^\top =-A$, respectively). A Hadamard matrix $H$ of order $n$ is \defn{skew} if $H+H^\top=2I_n$. A \defn{conference matrix of order $n$} is  an $n\times n$ 0-diagonal matrix $C$ with off-diagonal entries $\pm 1$  such that $CC^\top=(n-1)I_n$; that is, the rows of $C$ are pairwise orthogonal. 
It is easy to see that the order of a conference matrix must be  even, and  that $C$ is a skew-symmetric conference matrix if and only if $C+I$ is a skew  Hadamard matrix. Goethals and Seidel  showed that if $C$ is a conference matrix of order $n$, and $n\equiv 2\pmod{4}$ ($n\equiv 0\pmod{4}$, respectively), then there exist signed permutation matrices $D_1,D_2$ such that $D_1CD_2$ is symmetric (skew-symmetric, respectively) \cite{MR325418}.   
Constructions of symmetric conference matrices of order $n$ are known in the following cases \cite{BS14,MR1178508}. 
\begin{itemize}
    \item $n=p^r+1$ where $p$ is a prime satisfying $p\equiv1\pmod{4}$;  
    \item $n=q^2(q+2)+1$ where $q\equiv3\pmod{4}$ and both $q$ and $q+2$ are  prime powers;
    \item $n=5\cdot 9^{2t+1}+1$ where $t\geq 0$; 
    \item $n=(h-1)^s+1$ where $s\geq 2$ and $h$ is the order of a conference matrix;
    \item $n=(h-1)^{2s}+1$ where $s\geq 1$  and $h$ is the order of a skew-symmetric conference matrix.
\end{itemize}
Similarly, constructions of skew-symmetric conference matrices of order $n$
have been established in the following cases \cite{BS14,MR1178508}.
\begin{itemize}
    \item $n=2^t (q_1+1)\cdots (q_r+1)$ where $t\geq 0$  and  $q_i$ is a prime power satisfying $q_i\equiv 3\pmod{4}$ for $1\leq i\leq r$;
    \item $n=(h-1)^{2s-1}+1$ where $s\geq 1$ and $h$ is the order of a skew-symmetric conference matrix;
    \item $n=2(q+1)$ where $q$ is a prime power satisfying $p\equiv5\pmod{8}$, $q=p^t$,   $t\equiv 2\pmod{4}$;
    \item $n=4m$ where $m\in\{2t+1\ | \ 1\leq t\leq 15\}\cup\{37, 43, 67, 113, 127, 157, 163, 181, 213, 241, 631\}$.
\end{itemize}

 For a graph $\Gamma$ with adjacency matrix $A$, the \defn{Seidel (adjacency) matrix} $S$ is defined as $S=J-2A-I$;
Note that $S$ is a symmetric matrix over $\{0,\pm 1\}$ with zero diagonal entries and non-zero off-diagonal entries, and that if $S$ has two distinct eigenvalues $\pm\rho$, then it is a symmetric conference matrix of order $1+\rho^2$ \cite{MR2882891}.  

\subsection{Association schemes on triples}

We first define (commutative) association schemes \cite{BannaiIto1984}. 
Let $\{R_0,R_1,\ldots,R_m\}$ be a partition of $X^2$.
The pair $(X,\{R_i\}_{i=0}^m)$ is called
an {\em association scheme} of $m$ classes if the following conditions hold.
\begin{itemize}
\item $R_0 =\{(x,x) \mid x\in X\}$;
\item For any $i\in\{0,1,\ldots,m\}$, there exists some $j\in\{0,1,\ldots,m\}$ such that $\{(y,x) \mid (x,y)\in R_i\}=R_{j}$;
\item For any $i,j,k\in\{0,1,\ldots,m\}$ and any $(x,y)\in R_k$, the size of the set 
    \[
    \{z\in X \mid (x,z)\in R_i,(z,y)\in R_j\}
    \]
    depends only on $i,j,k$, and not  the choice of $(x,y)$.  
    This value is called the \defn{intersection number} and is denoted by $p_{ij}^{k}$;  
\item For any $i,j,k\in\{0,1,\ldots,m\}$, $p_{ij}^{k}=p_{ji}^k$. 
\end{itemize}
An association scheme is {\it symmetric} if  $R_1,\dots,R_m$ are symmetric. 

Now, we define association schemes on triples which were   originally
introduced by Mesner and Bhattacharya~\cite{MR1075708, MR1272106}.
Let $\{R_0,R_1,\ldots,R_m\}$ be
a partition of $X^3$. The pair
$\big(X,\{R_i\}_{i=0}^m\big)$ is an \emph{association scheme on triples} --- abbreviated as \emph{AST} --- of \emph{order} $n$ if the following  conditions hold.
\begin{enumerate}[label=(\Roman*)]
    \item\label{ast1} The relations $R_0,R_1,R_2,R_3$ are chosen so that
 \begin{align*}
    &&
    &R_0=\{(x,x,x)\mid x\in X\},
    &&
    R_1=\{(x,y,y)\mid x,y\in X,x\neq y\},\\
    &&&
    R_2=\{(x,y,x)\mid x,y\in X,x\neq y\},
    &&
    R_3=\{(x,x,y)\mid x,y\in X,x\neq y\};
    &&
  \end{align*}

    \item\label{ast2} For any $i\in\{0,1,\ldots,m\}$ and any permutation $\sigma\in S_3$,
    $\sigma(R_i)=R_j$ holds for some $j\in\{0,1,\ldots,m\}$; 
    \item\label{ast3} For any two distinct elements $y,z\in X$ and $i\in \{0,1,\ldots,m\}$, the size of the set $\{x\in X \mid (x,y,z)\in R_i\}$ depends only on $i$, and not on the choice of $y,z$;
    \item\label{ast4} For any $i,j,k,\ell\in\{0,1,\ldots,m\}$ and any triple $(x,y,z)\in R_\ell$, the size of the set
    \[
    \{w\in X \mid (w,y,z)\in R_i,(x,w,z)\in R_j,(x,y,w)\in R_k\}
    \]
    depends only on $i,j,k,\ell$, and not on the choice of $(x,y,z)$.
    This value is referred to as the \emph{intersection number} and is denoted by $p_{ijk}^{\ell}$.
\end{enumerate}
We shall refer to $R_0,\dots,R_3$ as the {\it trivial} relations. Regarding condition~\ref{ast3}, we write
\[
n_i^{(1)}=|\{x\in X \mid (x,y,z)\in R_i\}|
\]
for distinct $y,z\in X$. Similarly, we define
\[
n_i^{(2)}=|\{y\in X \mid (x,y,z)\in R_i\}|, \quad
n_i^{(3)}=|\{z\in X \mid (x,y,z)\in R_i\}|. 
\]
Because of condition~\ref{ast2}, these numbers also do not depend on the choice of
distinct elements $x,z\in X$ and $x,y\in X$, respectively. For $j\in [3]$, we have
that $n_0^{(j)}=n_j^{(j)}=0$ and $n_i^{(j)}=1$ for $i\in[3]\setminus \{j\}$.

An AST is  \emph{symmetric} when the relations $R_4,\dots,R_m$ are symmetric. In this case we have $n_i^{(1)} = n_i^{(2)} = n_i^{(3)} =: n_i$ for  $i \in \{4,\ldots,m\}$.

We represent each relation
$R_i$ by  an $n\times n\times n$ matrix~$A_i$ whose  entries are
\[
A_i(x,y,z)=\begin{cases}
    1 & \text{ if }(x,y,z)\in R_i,\\
    0 & \text{ otherwise}.
\end{cases}
\]
The matrix $A_i$ is the \emph{adjacency matrix} of the hypergraph $(X,R_i)$. 
Using the ternary product, condition~\ref{ast4}  can
be expressed as
\begin{align} \label{lincombast}
    A_i A_j A_k =\sum_{\ell=0}^m p_{ijk}^\ell A_\ell, \quad \mbox{ for  }i,j,k\in \{0,\ldots,m\}.
\end{align}
Moreover, condition (III) can
be restated as 
   \begin{align}\label{eq:jja}
    JJA_k=(\delta_{k0}+(v-1)\delta_{k3})(A_0+A_3)+n_k^{(3)}\left(J-A_0-A_3\right) , \quad \mbox{ for  }k\in \{0,\ldots,m\}.
    \end{align}
Observe that 
    \begin{align*}
      JJA_k=  \left(\sum_{i=0}^m A_i\right)\left(\sum_{j=0}^m A_j \right)A_k=\sum_{i,j=0}^m A_iA_jA_k=\sum_{\ell=0}^m\left(\sum_{i,j=0}^m p_{ijk}^\ell\right) A_\ell. 
    \end{align*}

\subsection{Association schemes on triples with two non-trivial classes}
For the purpose of this paper, we only need ASTs with six associate classes. 
We investigate its parameters. 
Let $\Omega=\big(X,\{R_i\}_{i=0}^5\big)$ be an association scheme on triples. 
An AST $(X,\{R_i\}_{i=0}^5)$ (where $R_0,\dots,R_3$ are trivial) is \defn{skew-symmetric} if 
it satisfies the following.
\begin{itemize}
    \item [(SS1)] $R_4$ and $R_5$ are symmetric with respect to the permutation $(123)\in S_3$;
    \item [(SS2)] $\sigma(R_4) \cap R_4=\sigma(R_5) \cap R_5=\emptyset$ for $\sigma:=(12)\in S_3$.    
\end{itemize} 
Condition (SS2) is equivalent to the following condition. 
\begin{itemize}
    \item [(SS2')] $\sigma(R_4) = R_5$ for $\sigma:=(12)\in S_3$.    
\end{itemize} 
The relations $R_4$ and $R_5$ are \defn{skew-symmetric} if they satisfy (SS1) and (SS2).  

Let $\{R_i\}_{i=0}^5$ be a partition of $X^3$ where $R_0,\ldots,R_3$ are trivial relations.  
To show that the pair $(\Omega,\{R_i\}_{i=0}^5)$ is an AST, we need to verify that the intersection numbers are well defined. 
Some intersection numbers can be derived formally, without using any special properties, as claimed in the next lemma.
We use the following notion in the next lemma.  
\begin{align*}
p_{ijk}^\ell(x,y,z)&=|\{w\in X \mid (w,y,z)\in R_i,\ (x,w,z)\in R_j,\ (x,y,w)\in R_k\}| \text{ for } (x,y,z)\in R_\ell,\\
n_i^{(1)}(y,z)&=|\{x\in X \mid (x,y,z)\in R_i\}| \text{ for distinct } y,z\in X,\\
n_i^{(2)}(x,z)&=|\{y\in X \mid (x,y,z)\in R_i\}| \text{ for distinct } x,z\in X,\\
n_i^{(3)}(x,y)&=|\{z\in X \mid (x,y,z)\in R_i\}| \text{ for distinct } x,y\in X.
\end{align*}
For $x\in\{4,5\}$, set $\dot x=9-x$.  
In order to prove a pair $(X,\{R_i\}_{i=0}^5)$ is an AST, we collect facts about intersection numbers, which hold for any $R_4,R_5$ symmetric or skew-symmetric. 
\begin{lemma}\label{lem:intersectionnumber}
    Let $\{R_i\}_{i=0}^5$ be a partition of $X^3$ where $R_0,\ldots,R_3$ are trivial relations.  
    Then the following holds. 
    \begin{enumerate}
        \item For $i,j,k\in\{0,1,2,3\}$ and $\ell\in\{0,1,\ldots,5\}$, intersection numbers $p_{ijk}^\ell$ are uniquely determined as follows. Intersection numbers not displayed below are zero. 
        \begin{align*}
            p_{000}^0&=1,\quad p_{011}^1=1,\quad         p_{202}^2=1,\quad p_{330}^3=1,\\         p_{132}^1&=1,\quad p_{321}^2=1,\quad p_{213}^3=1,\quad p_{123}^0=|X|-1.
        \end{align*}
        \item For $i,j,k$ such that exactly two of them belong to $\{0,1,2,3\}$ and the remaining one belongs to $\{4,5\}$, and for $\ell\in\{0,1,\ldots,5\}$, the intersection numbers $p_{ijk}^\ell$ are uniquely determined as follows. Intersection numbers not displayed below are zero. 
        \begin{align*} 
        p_{\ell32}^{\ell}&=p_{3\ell1}^{\ell}=p_{21\ell}^{\ell}=1 \text{ for } \ell\in\{4,5\}.
        \end{align*}
        \item For $i,j,k$ such that exactly one of them belongs to $\{0,1,2,3\}$ and the remaining two belong to $\{4,5\}$, and for $\ell\in\{0,1,\ldots,5\}$, the possible non-zero intersection numbers $p_{ijk}^\ell$, whenever they are well-defined, are as follows.
        \begin{align*}
        &p_{144}^1,\quad p_{145}^1,\quad p_{154}^1,\quad p_{155}^1,\\
        &p_{424}^2,\quad p_{425}^2,\quad p_{524}^2,\quad p_{525}^2,\\
        &p_{443}^3,\quad p_{453}^3,\quad p_{543}^3,\quad p_{553}^3.
        \end{align*}
        If $R_4$ and $R_5$ are symmetric, then 
        \begin{align*}
        p_{1ii}^1(x,y,z)&=n_i^{(1)}(y,z)\quad \text{ for }i\in\{4,5\},\\
        p_{i2i}^2(x,y,z)&=n_i^{(2)}(x,z)\quad \text{ for }i\in\{4,5\},\\
        p_{ii3}^3(x,y,z)&=n_i^{(3)}(x,y)\quad \text{ for }i\in\{4,5\}, \\
        p_{145}^1&=p_{154}^1=p_{425}^2=p_{524}^2=p_{453}^3=p_{543}^3=0. 
        \end{align*}
        If $R_4$ and $R_5$ are skew-symmetric, then 
        \begin{align*}
        p_{1i \dot i}^1(x,y,z)&=n_i^{(1)}(y,z)\quad \text{ for }i\in\{4,5\},\\
        p_{i2\dot i}^2(x,y,z)&=n_i^{(2)}(x,z)\quad \text{ for }i\in\{4,5\},\\
        p_{i\dot i 3}^3(x,y,z)&=n_i^{(3)}(x,y)\quad \text{ for }i\in\{4,5\}, \\
        p_{144}^1&=p_{155}^1=p_{424}^2=p_{525}^2=p_{443}^3=p_{553}^3=0.
        \end{align*}
        \item For $i,j,k$ such that all of them belong to $\{4,5\}$, and for $\ell\in\{0,1,\ldots,5\}$, the possible non-zero intersection numbers $p_{ijk}^\ell$, whenever they are well-defined, are as follows.
        \begin{align*}
        &p_{444}^4,\quad p_{445}^4,\quad p_{454}^4,\quad p_{455}^4,\quad p_{544}^4,\quad p_{545}^4,\quad p_{554}^4,\quad p_{555}^4,\\
        &p_{444}^5,\quad p_{445}^5,\quad p_{454}^5,\quad p_{455}^5,\quad p_{544}^5,\quad p_{545}^5,\quad p_{554}^5,\quad p_{555}^5.         
        \end{align*}
        \begin{enumerate}
            \item 
        If $R_4$ and $R_5$ are symmetric, then for $(x,y,z)\in R_\ell$, $\ell\in\{4,5\}$, 
        \begin{align*}
        &p_{445}^\ell(x,y,z)=p_{445}^\ell(y,x,z)=p_{454}^\ell(x,z,y)=p_{454}^\ell(y,z,x)=p_{544}^\ell(z,x,y)=p_{544}^\ell(z,y,x),\\
        &p_{455}^\ell(x,y,z)=p_{545}^\ell(y,x,z)=p_{455}^\ell(x,z,y)=p_{554}^\ell(y,z,x)=p_{545}^\ell(z,x,y)=p_{554}^\ell(z,y,x). 
        \end{align*}
        \item If $R_4$ and $R_5$ are skew-symmetric and let $(x,y,z)\in R_\ell$, $\ell\in\{4,5\}$, then 
        $(y,z,x),(x,z,y)\in R_\ell$ and $(x,z,y),(z,y,x),(y,x,z)\in R_{\dot \ell}$, and  
        \begin{align*}
        &p_{444}^\ell(x,y,z)=p_{555}^{\dot \ell}(y,x,z),\\        &p_{445}^\ell(x,y,z)=p_{554}^{\dot \ell}(y,x,z)=p_{545}^{\dot \ell}(x,z,y)=p_{454}^{\ell}(y,z,x)=p_{544}^\ell(z,x,y)=p_{455}^{\dot \ell}(z,y,x).
        \end{align*}
        \end{enumerate}
    \end{enumerate}
\end{lemma}
\begin{proof}
    All of them follow from simple set-theoretic arguments. We demonstrate only one case in (4) (b).  

    Let $R_4$ and $R_5$ be skew-symmetric, and let $(x,y,z)\in R_4$. Then $(y,x,z)\in R_5$ and 
    \begin{align*}
        &\{w\in X \mid (w,y,z)\in R_4,(x,w,z)\in R_4,(x,y,w)\in R_4\}\\
        &=\{w\in X \mid (w,x,z)\in R_5,(y,w,z)\in R_5,(y,x,w)\in R_5\}. 
    \end{align*}
    Comparing the size of these sets yields that $p_{444}^4(x,y,z)=p_{555}^5(y,x,z)$. 
\end{proof}

\subsection{Regular Two-Graphs and Symmetric ASTs}\label{sec:astrtt} 
Two-graphs are 3-uniform hypergraphs in which every 4-subset of points contains an even number of hyeredges; they should not be confused with 2-graphs in graph theory. Two-graphs play a crucial role in this paper. The study of two-graphs was initiated by Higman, motivated by investigations of Conway's sporadic simple group $\cdot  3$. These objects are also important in the study of strongly regular graphs, equiangular lines, and doubly transitive groups (see Taylor \cite{MR476587}). 

A \defn{two-graph} is a pair $\mathcal{T}:= (X,\Delta)$ where $\Delta\subseteq \binom{X}{3}$ (the \emph{coherent} triples)
and every $4$-subsets of $X$ contains an even number of  triples from $\Delta$. A two-graph  is $a$-\defn{regular} if every pair of distinct vertices is contained in exactly $a$ triples from $\Delta$. In other words, a two-graph $\mathcal{T}$ is $a$-regular if $\mathcal{T}$ is a 2-$(|X|, 3, a)$-design. From any graph $\Gamma=(X,E)$, one can define the \textit{associated}  two-graph $(X,\Delta)$ by 
\[
\Delta=\left\{\{x,y,z\}\in\tbinom{X}{3} \mid \text{the induced subgraph on $\{x,y,z\}$ has an odd number of edges}\right\}.
\]
Indeed, it is easy to see that any induced subgraph on $4$ vertices of $X$ contains $0$, $2$, or $4$ coherent triples. 
For a two-graph $(X,\Delta)$ and a vertex $w \in X$, let the \textit{descendent} graph $\Gamma_w=(X\setminus \{w\},E)$
where  two vertices $x,y$ are adjacent in $\Gamma_w$ if 
$\{w,x,y\}\in\Delta$.

Two graphs $\Gamma_1, \Gamma_2$ on the same set $X$ of vertices are \defn{switching-equivalent} if there is a subset $Y$ in $X$ such that the graph obtained by reversing the adjacency between $Y$ and $X\setminus Y$ in $\Gamma_1$ equals the graph $\Gamma_2$. 
Equivalence classes with respect to the switching-equivalence are said to be \defn{switching classes}.  
There is a one-to-one correspondence between a two-graph and a switching class of a graph \cite{MR2882891}.

A graph $\Gamma$ is said to be a \defn{strongly regular graph with parameters $(n,k,\lambda,\mu)$} if $|V(\Gamma)|=n$ and its adjacency matrix $A$ satisfies that $A^2=kI+\lambda A+\mu(J-I-A)$. 
\begin{theorem}\cite{MR2882891}\label{thm:rt-ast}
    For a graph $\Gamma=(X,E)$ with $n$ vertices, its associated two-graph $\mathcal{T}=(X,\Delta)$, and any  $w\in X$, the following are equivalent.
    \begin{enumerate}
        \item [\textup{(i)}] The Seidel matrix $S$ of $\Gamma$ has two distinct eigenvalues $\rho_1,\rho_2$. 
        \item [\textup{(ii)}] The two-graph $\mathcal{T}$ is $k$-regular. 
        \item [\textup{(iii)}] The graph $\Gamma_w$ is  strongly regular  with parameters $(n-1,k,\lambda,k/2)$. 
    \end{enumerate}
    The parameters are related by $n=1-\rho_1\rho_2,k=1-(\rho_1+1)(\rho_2+1)/4$ and $\lambda=1-(\rho_1+3)(\rho_2+3)/4$. The set $\Delta$ is related to the Seidel matrix $S$ by 
\[
\Delta=\big\{\{x,y,z\} \mid S(x,y)S(y,z)S(z,x)=-1\big\}. 
\]
\end{theorem}
In view of Theorem~\ref{thm:rt-ast}, a symmetric conference matrix of order $n$ yields an $(\frac{n+2}{4})$-regular two-graph.

\begin{theorem} \cite{MR1075708}\label{thm:astrt} 
 Let $(X,\Delta)$ be a regular two-graph. Let  $R_0,\ldots,R_3$  be  trivial relations, and
\begin{align*}
    R_4=\big\{(x,y,z)\in X^3\mid \{x,y,z\} \in \Delta\big\},\quad R_5=\big\{(x,y,z)\in X^3\mid \{x,y,z\} \in \tbinom{X}{3}\setminus \Delta\big\}. 
\end{align*}    
 Then  $(X,\{R_i\}_{i=0}^5)$   is a symmetric AST with the property that 
    \begin{align*}
    p_{444}^5=p_{445}^4=p_{455}^5=p_{555}^4=0.
    \end{align*}
\end{theorem}
The association schemes on triples in Theorem~\ref{thm:astrt} are characterized by regular two-graphs.  
\begin{theorem}\cite{MR1075708}\label{thm:astrt2} 
    Let $(X,\{R_i\}_{i=0}^5)$ be a symmetric AST
with the property that 
    \begin{align*}
    p_{444}^5=p_{445}^4=p_{455}^5=p_{555}^4=0.
    \end{align*}
   Then  $(X,\Delta)$ is a regular two-graph, where $\Delta=\big\{\{x,y,z\}\in\binom{X}{3} \mid (x,y,z)\in R_4\big\}$.  
\end{theorem}

\begin{remark}\label{rem:number_rt}\textup{
Let $(X,\Delta)$ be a regular two-graph with Seidel eigenvalues $\rho_1,\rho_2$. 
Then for any vertex $w\in X$, the descendent graph $\Gamma_w=(X\setminus\{w\},E)$ is strongly regular. 
Let 
\begin{align*}
    E_0&=\big\{(x,x) \mid x\in X\setminus\{w\}\big\},\\
    E_1&=\big\{(x,y) \mid \{x,y\}\in E\big\},\\
    E_2&=\big\{(x,y) \mid \{x,y\}\not\in E,x\neq y\big\}.
  \end{align*}
Then the pair $(X\setminus\{w\},\{E_i\}_{i=0}^2)$ is a symmetric association scheme. 
For $i=0,1,2$, let $k_i$ be the valency of the regular graph $(X\setminus\{w\},E_i)$, and let $p_{ij}^k$ denote the parameters of the scheme. 
Recall that $n=|X|$. 
The non-zero intersection numbers $p_{ijk}^\ell$ of the AST $(X,\{R_i\}_{i=0}^5)$ obtained in Theorem~\ref{thm:astrt} 
are as follows.
The intersection numbers $p_{ijk}^\ell$ for $i,j,k,\ell\in\{4,5\}$ are determined from the proof of \cite[Lemma~5.5]{MR1272106}.   
The indices of possible non-zero intersection numbers are restricted by \cite[Proposition~2.7] {MR1272106}. Their exact values, for example $p_{214}^4=1$, are routinely verified, and the other numbers, such as $p_{144}^1$, correspond to the valences of the associated strongly regular graphs.  
\[
\begin{aligned}
& p_{000}^0 = 1,\quad p_{123}^0 = n-1 = -\rho_1 \rho_2,\\[4pt]
& p_{011}^1 = 1,\quad p_{132}^1 = 1,\quad p_{144}^1 = k_1 = -\dfrac{(\rho_1+1)(\rho_2+1)}{2},\\
& \qquad p_{155}^1 = k_2 = -\dfrac{(\rho_1-1)(\rho_2-1)}{2},\\[4pt]
& p_{202}^2 = 1,\quad p_{321}^2 = 1,\quad p_{424}^2 = k_1 = -\dfrac{(\rho_1+1)(\rho_2+1)}{2},\\
& \qquad p_{525}^2 = k_2 = -\dfrac{(\rho_1-1)(\rho_2-1)}{2},\\[4pt]
& p_{213}^3 = 1,\quad p_{330}^3 = 1,\quad p_{553}^3 = k_2 = -\dfrac{(\rho_1-1)(\rho_2-1)}{2},\\[4pt]
& p_{214}^4 = 1,\quad p_{341}^4 = 1,\quad p_{432}^4 = 1,\quad p_{443}^4 = k_1 = -\dfrac{(\rho_1+1)(\rho_2+1)}{2},\\
& \qquad p_{444}^4 = p_{11}^1 = -\dfrac{\rho_1\rho_2 + 3\rho_1 + 3\rho_2 + 5}{4},\\
& \qquad p_{455}^4 = p_{12}^1,\quad p_{545}^4 = p_{12}^1,\quad p_{554}^4 = p_{12}^1,\\[4pt]
& p_{215}^5 = 1,\quad p_{351}^5 = 1,\quad p_{532}^5 = 1,\\
& \qquad p_{445}^5 = p_{12}^2,\quad p_{454}^5 = p_{12}^2,\quad p_{544}^5 = p_{12}^2,\\
& \qquad p_{555}^5 = p_{22}^2 = -\dfrac{\rho_1\rho_2 - 3\rho_1 - 3\rho_2 + 5}{4}.
\end{aligned}
\]
}\end{remark}

\section{Regular Skew Two-Graphs and Skew-Symmetric ASTs}\label{sec:skewtwograph}
In the previous section, we established a bijection between regular two-graphs and symmetric ASTs satisfying additional properties (Theorems~\ref{thm:astrt} and \ref{thm:astrt2}). In this section, we develop a one-to-one  correspondence between regular skew two-graphs and a class of skew-symmetric ASTs, beginning with the definitions of skew two-graphs, which are oriented analogues of two-graphs and were originally introduced by Cameron \cite{MR505778}.

Let $\mathcal C_3(X)$ (or $\mathcal C_3$ for short) be the set of all $3$-cycles in $S_X$. 
A \defn{skew two-graph} is a pair  $\Sigma:=(X,\nabla)$ where  $\nabla\subseteq\mathcal C_3$ satisfying the following.
\begin{itemize}
    \item [(S1)] For any $\tau\in \mathcal C_3$, exactly one of $\tau$ and $\tau^{-1}$ belongs to $\nabla$;
    \item [(S2)] For any $4$-subset $\{x,y,z,w\} \subseteq X$,  $\nabla$ contains an even number of the $3$-cycles $(xyz)$, $(xwy)$, $(xzw)$, and $(ywz)$.    
\end{itemize} 
For distinct  $x,y\in X$, the \defn{degree} of  $(x,y)$ in $\Sigma$  is the number of points $z\in X$ such that $(xyz)\in\nabla$, and $\Sigma$ is \defn{regular} if every pair $(x,y)$ has the same degree. 
If $\Sigma$ is regular, then its degree is always $(|X|-2)/2$.

Now, we review the equivalence between switching classes of tournaments and skew two-graphs. A \defn{tournament} is a directed graph $T=(X,E)$  such that for any distinct  $x,y\in V$, either $(x,y)\in E$ or $(y,x)\in E$. The 3-cycles of the skew two-graph $(X,\nabla)$ {\it associated} with $T$ are 
\[
\big\{(yzw) \in\mathcal C_3 \ \mid\  | E\cap \{(y,z),(z,w),(w,y)\}| \mbox{ is odd}\big\}.
\]
Verifying that $(X,\nabla)$ is a skew two-graph is  routine.  Conversely, given a skew two-graph $(X,\nabla)$ and $w\in X$, one can construct a tournament by defining the set $E_w$ by
\[
E_w=\big\{(y,z)\in (X\setminus\{w\})^2 \mid (wyz)\in \nabla \big\}. 
\]
By (S1), 
for any distinct $y,z\in X\setminus\{w\}$, exactly one of $(y,z),(z,y)$ belongs to $E_w$.  
Then, $T_w:=(X\setminus\{w\},E_w)$ is a \textit{descendent} tournament. 
Consider the tournament $T=(X,E)$ where $E=E_w\cup\{
(w,y) \mid y\in X\setminus\{w\}\}$.   
A \defn{switching} with respect to a subset $Y$ of the vertex set $X$ of a tournament $\mathcal{T}$ is a tournament obtained by reversing the orientation of the arcs between $Y$ and $X\setminus Y$. 
Switching gives rise to an equivalence relation on  tournaments, and the equivalence classes are called \defn{switching classes}.  
The \defn{Seidel matrix} of a tournament is a matrix $S$ indexed by the elements of $X$ such that $S(x,x)=0$ for any $x\in X$, $S(x,y)=1$ if $(x,y)\in E$, and $S(x,y)=-1$ if $(y,x)\in E$. 
Switching does not change the value $S(x,y)S(y,z)S(z,x)$ for any distinct $x,y,z\in X$. 
Therefore switching-equivalent tournaments have the same associated skew two-graph. It is shown in \cite{MR1344584} that there is a one-to-one correspondence between skew two-graphs and switching classes of tournaments.

A tournament is \defn{doubly regular} if the size of common out-neighbors of any two vertices $x,y$ depends only on whether  or not $x=y$. In an $n$-vertex doubly regular tournament, the number of out-neighbors of each vertex is $(n-1)/2$, and the number of common out-neighbors of any two distinct vertices is $(n-3)/4$. If $A$ is the adjacency matrix of a doubly regular tournament, then the diagonal, and off-diagonal entries of  $A A^\top$ are the number of out-neighbors of each vertex, and  the number of common out-neighbors of every pair of distinct vertices, respectively. In other words,
$$
A A^\top= \dfrac{n-1}{2} I + \dfrac{n-3}{4} (J-I).
$$
Conversely, if an $n\times n$ matrix $A$ satisfies the above condition, it corresponds to an $n$-vertex doubly regular tournament.

A skew-symmetric conference matrix $C$ (or, equivalently a skew Hadamard matrix) and a doubly regular tournament whose adjacency matrix is $A$, are related as follows.  
Let $C$ be an $n\times n$ zero-diagonal  skew-symmetric  matrix whose off-diagonal entries  are $\pm 1$. Let $D$ be the diagonal matrix with $D(1,1)=1$ and  $D(i,i)=C(1,i)$ for $2\leq i\leq n$. Then $DCD$ can be expressed in the form 
\[
C \;=\;
\begin{pmatrix}
0 & \mathbf{1}^\top \\
-\mathbf{1} & S
\end{pmatrix}
\]
where ${\vb 1}$ is the all-ones column vector and $S$ is the  Seidel  matrix of a tournament.
 By \cite[Theorem~2]{reid1972doubly},  $C$ is a skew-symmetric conference matrix if and only if $S$ is the Seidel matrix of a doubly regular tournament.

It is known  that for a skew-symmetric conference matrix $C$ whose rows and columns are indexed by a set $X$, and
\[
\nabla:=\left\{(xyz) \in \mathcal C_3\mid C(x,y)C(y,z)C(z,x)=-1\right\}, 
\]
$(X,\nabla)$ is a regular skew two-graph (see \cite[Section~3]{MR1344584}). 

\begin{lemma}\label{lem:regularskewtwograph} 
For a skew two-graph $(X,\nabla)$  and  $w\in X$,  the following are equivalent. 
    \begin{enumerate}
        \item [\textup{(i)}] The skew two-graph $(X,\nabla)$ is  regular. 
        \item [\textup{(ii)}] The descendent tournament $T_w$ is doubly regular. 
    \end{enumerate}
\end{lemma}
\begin{proof}
    Set $n=|X|$. 
   First, let us assume that (i)  holds.     Since $(X,\nabla)$ is regular, 
    the size of the following set 
    \[
    \{z\in X \mid (wyz)\in\nabla\}
    \]
    for $y\neq x$ equals the degree, $(n-2)/2$ where $n=|X|$, which is also equal to the number of edges $(y,z)$ in $T_w$. 
    Therefore $T_w$ is a regular tournament for any $w$. 
    Now, we fix two distinct vertices $y,z$ in $T_w$ and we may assume that $(y,z)\in E_w$. Let
    \begin{align*}
        N_w(x)&=\{u\in X\setminus\{w\} \mid (x,u)\in E_w\},\\
        \overline{N}_w(x)&=\{u\in X\setminus\{w\} \mid (x,u)\not\in E_w\}, 
    \end{align*}
    for $x\in X\setminus\{w\}$. 
    Since $\mathcal{T}_w$ is regular, 
    \begin{align*}
        |N_w(x)|=|\overline{N}_w(x)|=\frac{n-2}{2}    
    \end{align*}
    for any $x\in X\setminus\{w\}$. 
    Since $|N_w(y)|=1+|N_w(y)\cap N_w(z)|+|N_w(y)\cap \overline{N}(z)|$ and $|\overline{N}_w(y)|=|\overline{N}_w(y)\cap N(z)|+|\overline{N}_w(y)\cap \overline{N}_w(z)|$,  we have 
    \begin{align}\label{eq:drt1}
        1+|N_w(y)\cap N_w(z)|+|N_w(y)\cap \overline{N}_w(z)|=|\overline{N}_w(y)\cap N_w(z)|+|\overline{N}_w(y)\cap \overline{N}_w(z)|=\frac{n-2}{2}.  
    \end{align}
    Similarly, by considering $N_w(z)$, 
    \begin{align}\label{eq:drt11}
        |N_w(y)\cap N_w(z)|+|\overline{N}_w(y)\cap N_w(z)|=\frac{n-2}{2}.  
    \end{align}

    We now consider another descendent tournament $T_y$. Since $T_y$ is obtained from the tournament $T_w \cup\{w\}$, which is a tournament obtained from $T_w$ by adding a dominating vertex $w$, by switching with respect to $\{y\}\cup N_w(y)$ and deleting the resulting dominating vertex $y$.  
    We now apply the same argument to this tournament $T_y$ as above. 
    Since $T_y$ is regular, 
    \begin{align*}
        |N_y(z)|=\frac{n-2}{2}.     
    \end{align*}
    On the other hand, it holds that 
    \begin{align*}
        N_y(z)&=(\overline{N}_w(y)\cap N_w(z))\cup(N_w(y)\cap \overline{N}_w(z)).  
    \end{align*}
    Therefore we have 
    \begin{align}\label{eq:drt2}
        |\overline{N}_w(y)\cap N_w(z)|+|N_w(y)\cap \overline{N}_w(z)|=\frac{n-2}{2}.  
    \end{align}
    Combining \eqref{eq:drt1}, \eqref{eq:drt11}, \eqref{eq:drt2} yields that 
    \begin{align*}
        |N_w(y)\cap N_w(z)|&=\frac{n-4}{4},\\
        |N_w(y)\cap \overline{N}(z)|&=\frac{n-4}{4},\\
        |\overline{N}_w(y)\cap N(z)|&=\frac{n}{4},\\
        |\overline{N}_w(y)\cap \overline{N}_w(z)|&=\frac{n-4}{4}.
    \end{align*}
    Therefore, $|N_w(y)\cap N_w(z)|$ does not depend on the choice of $y,z$, and (ii) holds. 
    
    (ii) $\Rightarrow$ (i): 
    Assume (ii) holds. Each tournament $T_w$ is regular. Therefore, the valency of the skew two-graph is constant, (i) holds.   
\end{proof}

Recall that an AST $(X,\{R_i\}_{i=0}^5)$ (where $R_0,\dots,R_3$ are trivial) is skew-symmetric if 
it satisfies the following.
\begin{itemize}
    \item [(SS1)] $R_4$ and $R_5$ are symmetric with respect to the permutation $(123)\in S_3$;
    \item [(SS2)] $\sigma(R_4) \cap R_4=\sigma(R_5) \cap R_5=\emptyset$ for $\sigma:=(12)\in S_3$.    
\end{itemize}

\begin{theorem}\label{thm:astrst}
    Let $(X,\nabla)$ be a regular skew two-graph. Let $R_0,\ldots,R_3$ be trivial relations, and 
    \begin{align*}
    R_4=\{(x,y,z)\in X^3\mid (xyz) \in \nabla\},\quad R_5=\{(x,y,z)\in X^3\mid (xyz) \in \mathcal C_3\setminus \nabla\}. 
\end{align*}
    Then  $(X,\{R_i\}_{i=0}^5)$  is a skew-symmetric AST with the property that 
    \[p_{444}^4=p_{445}^5=p_{455}^4=p_{555}^5=0.\]  
\end{theorem}
\begin{proof}
    Observe that  $R_4$ and $R_5$ are symmetric with respect to the permutation $(123)$ but not  with respect to the permutation $(12)$.   Since  $\{R_0,\ldots,R_5\}$ partitions $X^3$ and condition (I) in the definition of an AST holds, we  verify that conditions (II)--(VI). 
    Since the elements of $R_4, R_5$ correspond to 3-cycles, condition (S1) in the definition of skew two-graphs ensures that condition (II) is satisfied.
    Since  $(X,\nabla)$ is regular, condition (III)  is satisfied. To verify condition (VI),  let $n=|X|$. We show that 
    \[
    A_iA_jA_k\in\mathcal{A}:=\text{Span}\{A_0,\ldots,A_5\} \quad \mbox{ for }i,j,k\in\{0,1,\ldots,5\}.
    \]
    By Lemma~\ref{lem:intersectionnumber}, it is sufficient to show that 
    \[
    A_4A_4A_4\in \text{span}\{A_5\} \text{ and }A_4A_4A_5\in \text{span}\{A_4\}
    \]
    For $(x,y,z)\in R_5$, by the $3$-$(n,4,n/4)$ design $\mathfrak{H}(n,4,9)$ in \cite[Example~2.4]{greaves2024constructions}, we have 
    \begin{align*}
(A_4A_4A_4)(x,y,z)
&= \left|\{w\in X \mid \text{the induced subtournament on }\{x,y,z,w\}\text{ in }T\text{ is a diamond}\}\right| \\
&= \frac{n}{4}.
\end{align*}
where $ T:=T_w \cup\{w\}$ is a tournament obtained from $T_w$ by adding a dominating vertex $w$ and a diamond is a 4-vertex tournament in which one vertex either dominates or is dominated by a directed 3-cycle. Thus, 
    \begin{align}\label{eq:444}     
    A_4A_4A_4=\frac{n}{4} A_5. 
    \end{align}
    Next we calculate $A_4A_4J$. Its $(x,y,z)$-entry for $(x,y,z)\in R_i$, $i\in\{0,1,2,3\}$ is zero by Lemma~\ref{lem:intersectionnumber}. Let $(x,y,z)\in R_\ell$ with $\ell \in\{4,5\}$. 
    Then $(z,x,y)\in R_\ell$, which implies that $(x,y)\in E_w$ if $\ell=4$ and $(x,y)\not\in E_w$ if $\ell=5$ in the tournament $T_w$. 
    By the proof of Lemma~\ref{lem:regularskewtwograph}, 
    \begin{align*}
        (A_4A_4J)(x,y,z)&=|\{w\in X \mid (w,y,z)\in R_4,(x,w,z)\in R_4\}|\\
        &=|\{w\in X\setminus\{z\} \mid (z,y,w)\in R_5,(z,x,w)\in R_4\}|\\
        &=|\overline{N}_w(y)\cap N_w(x)|\\
        &=\begin{cases}
             \frac{n-4}{4} &\text{ if } \ell=4,\\
             \frac{n}{4} &\text{ if } \ell=5.
        \end{cases}
    \end{align*}
    Therefore, $A_4A_4(A_4+A_5)=A_4A_4J=\frac{n-4}{4}A_4+\frac{n}{4}A_5$.  
    By \eqref{eq:444}, 
    \begin{align}\label{eq:445}
    A_4A_4A_5=\left(\dfrac{n}{4}-1\right)A_4.
    \end{align}
    By Lemma~\ref{lem:intersectionnumber}, the remaining intersection numbers are uniquely determined. 
     This completes the proof.  
\end{proof}

\begin{theorem}
    Let $(X,\{R_i\}_{i=0}^5)$ be a skew-symmetric AST where $R_0,\ldots,R_3$ are trivial, and that\[p_{444}^4=p_{445}^5=p_{455}^4=p_{555}^5=0.\] 
    Then $(X,\nabla)$ is a regular skew two-graph, where $\nabla=\{(xyz)\in \mathcal C_3\mid (x,y,z) \in R_4\}.$  
\end{theorem}
\begin{proof}
Condition (III) in the definition of an AST ensures that $(X,\nabla)$ is regular. Let $\{x,y,z,w\}$ be a $4$-subset of $X$, and  $T=\{(xyz),(xwy),(xzw),(ywz)\}, T_4=T\cap \nabla, T_5=T\setminus T_4$. It suffices to show that $|T_4|$ is even.   
Suppose to the contrary that $|T_4|\in\{1,3\}$. If $|T_4|=3$, there are four cases to consider.
\begin{itemize}
    \item $(xyz)\in T_5$: We have  $(x,y,z)\in R_5$ and $(x,w,y),(x,z,w),(y,w,z)\in R_4$. Using (SS2), $(x,y,w),(x,w,z),(w,y,z)\in R_5$; together with $(x,y,z)\in R_5$, this contradicts $p_{555}^5=0$.
    \item $(xwy)\in T_5$: We have $(x,w,y)\in R_5$ and $(x,y,z),(x,z,w),(y,w,z)\in R_4$. Using (SS2), $(x,y,w)\in R_4$ and $(x,w,z),(w,y,z)\in R_5$;
    together with $(x,y,z)\in R_4$, this contradicts $p_{455}^4=0$. 
    \item $(xzw)\in T_5$: We have $(x,z,w)\in R_5$ and $(x,y,z),(x,w,y),(y,w,z)\in R_4$. Using (SS2), $(x,y,z),(x,w,z)\in R_4$ and $(x,y,w),(w,y,z)\in R_5$;
     this contradicts $p_{545}^4=p_{455}^{4}=0$. 
    \item $(ywz)\in T_5$: We have $(y,w,z)\in R_5$ and $(x,y,z),(x,w,y),(x,z,w)\in R_4$. Using (SS2), $(w,y,z)\in R_4$ and $(x,y,w),(x,w,z)\in R_5$; 
    together with $(x,y,z)\in R_4$, this contradicts $p_{455}^{4}=0$. 
\end{itemize}
If $|T_4|=1$, the results follow by replacing $T_4, T_5, R_4,R_5,p_{555}^5,p_{545}^4=p_{455}^4$ with $T_5, T_4, R_5,R_4,p_{444}^4$, $p_{454}^5=p_{445}^5$, respectively.
\end{proof}

\begin{remark}\label{rem:number_srt}
\textup{The non-zero intersection numbers $p_{ijk}^k$ of the AST $(X,\{R_i\}_{i=0}^5)$ obtained in Theorem~\ref{thm:astrst} 
are as follows. 
The indices of possible non-zero intersection numbers are restricted by \cite[Proposition~2.7] {MR1272106}. Their exact values, for example $p_{214}^4=1$, are routinely verified, and the other numbers, such as $p_{145}^1$, correspond to the valences of the associated doubly regular tournaments.  
\[
\begin{aligned}
& p_{000}^0 = 1,\quad p_{123}^0 = n-1,\\[4pt]
& p_{011}^1 = 1,\quad p_{132}^1 = 1,\quad p_{145}^1 = \dfrac{n}{2}-1,\quad p_{154}^1 = \dfrac{n}{2}-1,\\[4pt]
& p_{202}^2 = 1,\quad p_{321}^2 = 1,\quad p_{425}^2 = \dfrac{n}{2}-1,\quad p_{524}^2 = \dfrac{n}{2}-1,\\[4pt]
& p_{213}^3 = 1,\quad p_{330}^3 = 1,\quad p_{453}^3 = \dfrac{n}{2}-1,\quad p_{543}^3 = \dfrac{n}{2}-1,\\[4pt]
& p_{214}^4 = 1,\quad p_{341}^4 = 1,\quad p_{432}^4 = 1,\\
& \qquad p_{445}^4 = \dfrac{n}{4}-1,\quad p_{454}^4 = \dfrac{n}{4}-1,\quad p_{544}^4 = \dfrac{n}{4}-1,\quad p_{555}^4 = \dfrac{n}{4},\\[4pt]
& p_{215}^5 = 1,\quad p_{351}^5 = 1,\quad p_{532}^5 = 1,\\
& \qquad p_{444}^5 = \dfrac{n}{4},\quad p_{455}^5 = \dfrac{n}{4}-1,\quad p_{545}^5 = \dfrac{n}{4}-1,\quad p_{554}^5 = \dfrac{n}{4}-1.
\end{aligned}
\]
}\end{remark}
\begin{example}\textup{
    An example of ASTs in Theorem~\ref{thm:astrt} for $X=\{1,2,3,4\}$ is:
    \begin{align*}
        R_4=\{\sigma(x) \mid x\in Y_1,\sigma \in S_3\},\quad 
        R_5=\{\sigma(x) \mid x\in Y_2,\sigma \in S_3\},
    \end{align*}
    where 
    \begin{align*}
        Y_1&=\{(1, 2, 3), (1, 3, 4), (1, 4, 2), (2, 1, 4), (2, 3, 1), (2, 4, 3), (3,
   1, 2), (3, 2, 4), \\
   &\quad\quad(3, 4, 1), (4, 1, 3), (4, 2, 1), (4, 3, 2)\},\\
   Y_2&=\{(1, 2, 4), (1, 3, 2), (1, 4, 3), (2, 1, 3), (2, 3, 4), (2, 4, 1), (3,
   1, 4), (3, 2, 1), \\
   &\quad\quad(3, 4, 2), (4, 1, 2), (4, 2, 3), (4, 3, 1)\}.
    \end{align*}
}\end{example}

\section{Hadamard cubes from association schemes on triples}\label{sec:hcast}
In this section, we construct Hadamard cubes from ASTs.  
Recall that in an  $(n,d,\ell)$-Hadamard hypercube,  every two parallel $\ell$-layers are orthogonal. Let 
\[
\Omega_1 \text{ be the AST from Theorem~\ref{thm:astrt} with } \rho_1+\rho_2=0, \quad 
\Omega_2 \text{ be the AST from Theorem~\ref{thm:astrst}}.
\]
\begin{theorem}\label{thm:hcube}
Let $\Omega \in \{\Omega_1,\Omega_2\}$ and $H = \sum_{i=0}^5 a_i A_i$ with $a_0,\dots,a_5 \in \{\pm 1\}$ where $A_0,\dots,A_5$ are the adjacency matrices of $\Omega$. Then $H$ is an $(|X|,3,2)$-Hadamard cube if and only if $a_0 a_1 a_2 a_3 = a_4 a_5 = -1$. 
\end{theorem}
\begin{proof}
To determine the conditions on $a_0,\ldots,a_5$ that ensure $H$ is a Hadamard cube, we calculate the following sums
\begin{align*}
\mathscr{S}_1:= \mathscr{S}_1(y,\dot y) &=  \sum_{x,z\in X} H(x,y,z)H(x,\dot{y},z)  \quad \text{for distinct }y,\dot y\in X,\\
\mathscr{S}_2:= \mathscr{S}_2(z,\dot z) &=  \sum_{x,y\in X} H(x,y,z)H(x,y,\dot{z})  \quad \text{for distinct }z,\dot z\in X,\\
\mathscr{S}_3:= \mathscr{S}_3(x,\dot x) &= \sum_{y,z\in X} H(x,y,z)H(\dot{x},y,z) \quad \text{for distinct }x,\dot x\in X,
\end{align*} in the framework of the ternary algebra of $\Omega$.

By condition (II), for every $\sigma\in S_3$, there exists  some permutation $(\overline a_i)_{i=0}^{5}$ of $(a_i)_{i=0}^{5}$ such that 
\begin{align} \label{sigmacube}
    H^{\sigma}=\sum_{i=0}^5 a_i A_i^\sigma=\sum_{i=0}^5 \overline a_i A_i.
\end{align}
Since  $A_0,A_1,A_2,A_3$ are the adjacency matrices of the trivial relations in $\Omega$, $A_0^{\sigma}=A_0$. Moreover, if $\Omega$ is  symmetric, then $A_4^{\sigma}=A_4,A_5^{\sigma}=A_5$, and if $\Omega$ is  skew-symmetric and $\sigma$ is a transposition, then $A_4^{\sigma}=A_5, A_5^{\sigma}=A_4$. Finally, $A_1^{\sigma_1}=A_1, A_2^{\sigma_1}=A_3, A_3^{\sigma_1}=A_2$, $A_1^{\sigma_2}=A_3, A_2^{\sigma_2}=A_2, A_3^{\sigma_2}=A_1$, and $A_1^{\sigma_3}=A_2, A_2^{\sigma_3}=A_1, A_3^{\sigma_3}=A_3$ where
\[
\sigma_1 := (23), \quad \sigma_2 := (13), \quad \sigma_3 := (12).
\]
Thus, in connection to \eqref{sigmacube}, we have the following.
 \begin{align} \label{perma_i}
        (\overline a_i)_{i=0}^{5}&=\begin{cases}
        (a_0,a_1,a_3,a_2,a_4,a_5) & \text{ if }\sigma=\sigma_1, \Omega=\Omega_1, \\
        (a_0,a_1,a_3,a_2,a_5,a_4) &\text{ if }\sigma=\sigma_1, \Omega=\Omega_2, \\
        (a_0,a_3,a_2,a_1,a_4,a_5) & \text{ if }\sigma=\sigma_2, \Omega=\Omega_1, \\
        (a_0,a_3,a_2,a_1,a_5,a_4) & \text{ if }\sigma=\sigma_2, \Omega=\Omega_2, \\
         (a_0,a_2,a_1,a_3,a_4,a_5) & \text{ if }\sigma=\sigma_3, \Omega=\Omega_1, \\
        (a_0,a_2,a_1,a_3,a_5,a_4) & \text{ if }\sigma=\sigma_3, \Omega=\Omega_2. \\
\end{cases}
\end{align}

For $i,j,k,\ell\in \{0,\dots,5\}$, let  $^1 c_{ijk}^\ell=\overline a_j a_k\ p_{ijk}^\ell, ^2 c_{ijk}^\ell= a_i \overline a_k\ p_{ijk}^\ell, ^3 c_{ijk}^\ell=\overline a_i a_j\ p_{ijk}^\ell$. Expanding the ternary product $H^\sigma HJ$ in the Bose-Mesner algebra yields
    \begin{align*}
        H^{\sigma}HJ=\left(\sum_{i=0}^5 \overline a_i A_i\right)\left(\sum_{j=0}^5 a_jA_j\right)\left(\sum_{k=0}^5 A_k\right)&=\sum_{i,j,k=0}^5 \overline a_i a_j\ A_iA_jA_k
        =\sum_{\ell=0}^5\left(\sum_{i,j,k=0}^5 {^3 c_{ijk}^\ell} \right) A_\ell.
    \end{align*}
Similarly, we have $JH^{\sigma}H=\sum_{\ell=0}^5\left(\sum_{i,j,k=0}^5 {^1 c_{ijk}^\ell} \right) A_\ell$, and $HJH^{\sigma}=\sum_{\ell=0}^5\left(\sum_{i,j,k=0}^5 {^2 c_{ijk}^\ell} \right) A_\ell$. 

We claim that 
\begin{align*}
     \mathscr  S_t=\sum_{i,j,k,\ell=0}^5 {^t c_{ijk}^\ell} n_\ell^{(t)} \quad \text{for } t\in[3].
\end{align*}
To prove this claim, observe that for distinct $y,\dot y\in X$, we have
\begin{align*}
\mathscr{S}_1
&= \sum_{x,z\in X} J(z,\dot{y},y)\, H^{\sigma_1}(x,z,y)\, H(x,\dot{y},z) \\
&= \sum_{x\in X} (J H^{\sigma_1} H)(x,\dot{y},y)
&& \text{by \eqref{ternaryprod}} \\
&= \sum_{\ell=0}^5 \Bigl(\sum_{i,j,k=0}^5 {^1 c_{ijk}^\ell}\Bigr)
    \sum_{x\in X} A_\ell(x,\dot{y},y) \\
&= \sum_{\ell=0}^5 \Bigl(\sum_{i,j,k=0}^5 {^1 c_{ijk}^\ell}\Bigr)
    n_\ell^{(1)}
&& \text{by (III)} \\
&= \sum_{i,j,k,\ell=0}^5 {^1 c_{ijk}^\ell}\, n_\ell^{(1)}.
\end{align*}
 
Similarly, for distinct $z,\dot z\in X$, we have 
\begin{align*}
\mathscr{S}_2
&= \sum_{x,y\in X} H(x,y,z)\, J(\dot{z},x,z)\, H^{\sigma_2}(\dot{z},y,x) \\
&= \sum_{y\in X} (H J H^{\sigma_2})(\dot{z},y,z)
&& \text{by \eqref{ternaryprod}} \\
&= \sum_{\ell=0}^5 \Bigl(\sum_{i,j,k=0}^5 {^2 c_{ijk}^\ell}\Bigr)
    \sum_{y\in X} A_\ell(\dot{z},y,z) \\
&= \sum_{\ell=0}^5 \Bigl(\sum_{i,j,k=0}^5 {^2 c_{ijk}^\ell}\Bigr)
    n_\ell^{(2)}
&& \text{by (III)} \\
&= \sum_{i,j,k,\ell=0}^5 {^2 c_{ijk}^\ell}\, n_\ell^{(2)}.
\end{align*}

To complete the proof of our claim,  note that for   distinct $x,\dot x\in X$, we have
\begin{align*}
\mathscr{S}_3
&= \sum_{y,z\in X} H^{\sigma_3}(y,x,z)\,H(\dot{x},y,z)\,J(\dot{x},x,y) \\
&= \sum_{z\in X} (H^{\sigma_3} H J)(\dot{x},x,z)
&& \text{by \eqref{ternaryprod}} \\
&= \sum_{\ell=0}^5 \Bigl(\sum_{i,j,k=0}^5 {^3 c_{ijk}^\ell}\Bigr)
    \sum_{z\in X} A_\ell(\dot{x},x,z) \\
&= \sum_{\ell=0}^5 \Bigl(\sum_{i,j,k=0}^5 {^3 c_{ijk}^\ell}\Bigr)
    n_\ell^{(3)}
&& \text{by (III)} \\
&= \sum_{i,j,k,\ell=0}^5 {^3 c_{ijk}^\ell}\, n_\ell^{(3)}.
\end{align*}

To finish the proof, we use \eqref{perma_i}, and the intersection numbers in Remarks ~\ref{rem:number_rt},  \ref{rem:number_srt} to
simplify $\mathscr S_1,\mathscr S_2,\mathscr S_3$ and show that 
 \begin{align*}
       \mathscr S_1 =\mathscr S_2=\mathscr S_3=0 \text{ if and only if } a_0a_1a_2a_3=a_4a_5=-1.
    \end{align*}

Recall that $n_0^{(j)}=n_j^{(j)}=0$ and $n_i^{(j)}=1$ for $i\in[3]\setminus \{j\}$.
Note that $n_4^{(j)}=n_5^{(j)}=(n-2)/2$ for $j\in [3]$. 

   If $\Omega=\Omega_1$,  we have     
    \begin{align*}
    \mathscr S_1&=\frac{(n-1)^2}{2}(1+a_4  a_5) + 2 (a_0  a_2  +  a_1  a_3)  +(n-1)(a_1 +a_3)(a_4   +a_5),\\
    \mathscr S_2&=\frac{(n-1)^2}{2}(1+a_4  a_5) + 2 (a_0  a_3  + a_1  a_2)  +(n-1)(a_1 +a_2)(a_4   +a_5),\\
    \mathscr S_3 &=\frac{(n-1)^2}{2}(1+a_4  a_5) + 2 (a_0  a_1  +  a_2  a_3) +(n-1) (a_2 +a_3)(a_4   +a_5).
    \end{align*}
If $a_4 a_5=1$, then $\mathscr S_1=(n-1)^2 + 2 (a_0  a_2  +  a_1  a_3)  +2(n-1)a_4(a_1 +a_3)\not\equiv 0 \pmod 2$ for $n$ is even. So, we must have $a_4 a_5=-1$, and so $a_4+a_5=0$. Now, we have $\mathscr S_1 = 2(a_0 a_2 + a_1 a_3), \mathscr S_2 = 2(a_0 a_3 + a_1 a_2),\mathscr S_3 = 2(a_0 a_1 + a_2 a_3)$. 
Hence, $\mathscr S_1 =\mathscr S_2=\mathscr S_3=0$ if and only if $a_4 a_5=-1$ and 
$a_0 a_2 = -a_1 a_3, a_0 a_3 =- a_1 a_2, a_0 a_1 =- a_2 a_3$.    Since $a_i^2=1$, this is equivalent to $a_0a_1a_2a_3=a_4a_5=-1$.  
    
    If $\Omega=\Omega_2$,  we have 
    \begin{align*}
    \mathscr S_1&=
\frac{(n-2)^2}{2}(1+a_4 a_5) + 2 (a_0  a_2  + a_1  a_3)  +(n- 1) (a_1 + a_3)(  a_4 + a_5),\\
    \mathscr S_2&=\frac{(n-2)^2}{2}(1+a_4 a_5) + 2 (a_0  a_3 + a_1  a_2)  +(n- 1) (a_1 + a_2)(  a_4 + a_5),\\
    \mathscr S_3&=\frac{(n-2)^2}{2}(1+a_4  a_5) + 2 (a_0  a_1  +  a_2  a_3)  +(n- 1) (a_2 + a_3)(  a_4 + a_5).
    \end{align*}
Recall that $n\equiv 0 \pmod 4$. If $a_4 a_5=1$ and $n\geq 8$, then let $r=n-2\geq 6$. Then $\mathscr S_1=r^2+ 2a_4 (a_1+a_3)(r+1) + 2(a_0 a_2+a_1 a_3)\geq r^2 -4(r+1)-4=r^2-4r-8>0$. 
For the case $n=4$, assume that $\mathscr S_1=\mathscr S_2=\mathscr S_3=0$. Then, simplifying $\mathscr S_1=\mathscr S_2,\mathscr S_1=\mathscr S_3,\mathscr S_2=\mathscr S_3$ respectively, we have
\begin{align*}
    (a_2-a_3)(a_0-a_1-3a_4)&=0,\\
    (a_1-a_2)(a_0-a_3-3a_4)&=0,\\
    (a_1-a_3)(a_0-a_2-3a_4)&=0.
    \end{align*}
    Since $a_0-a_1-3a_4\neq0,a_0-a_3-3a_4\neq0$, and  $a_0-a_2-3a_4\neq0$, we have $a_1=a_2=a_3$. 
    Then 
    \begin{align*}
    \mathscr S_1=\mathscr S_2=\mathscr S_3=6+ 2 a_0  a_1   +12a_1a_4,
    \end{align*}
    which cannot be zero. 
If $a_4 a_5=-1$, then the remaining steps of the computation proceed exactly as in the case of $\Omega_1$. 
\end{proof}
\begin{corollary}
    Let $\mathbb{F}_q$ be the finite field of order $q$, where $q$ is odd. 
    Let $C$ be a conference matrix defined by 
    \[
C =
\begin{pmatrix}
0 & \mathbf{1}^\top \\
\pm \mathbf{1} & \chi(x-y)
\end{pmatrix}.
\]
   where $\chi$ is the quadratic character of $\mathbb{F}_q$, $x,y\in\mathbb{F}_q$, and $(a_i)_{i=0}^5=(-1,1,1,1,1,-1)$. Then the resulting Hadamard cube obtained by Theorem~\ref{thm:hcube} is the same as the one in Theorem~\ref{thm:paleycube}.
\end{corollary}

\section{Hadamard Hypercubes from Latin Hypercubes}\label{sec:revursive}
In this section, we construct higher-dimensional Hadamard cubes from lower-dimensional Hadamard cubes and Latin hypercubes.

We take $x \pmod n$ to lie in $[n]$, whereas $[x]_n$ denotes the unique integer $y \in \{0,\dots,n-1\}$ such that $y \equiv x \pmod n$.   
A \defn{Latin square} of order $n$ is an $n\times n$ array with entries $1,\ldots,n$ such that each row and each column contains each symbol in $[n]$.  Let $A$ be a $d$-dimensional hypercube of order $n$. 
For $i\in [d]$ and $a\in [n]$, let $A_{i;a}$ denote the $(d-1)$-layer of $A$ obtained by fixing the $i$-th coordinate to be $a$, while allowing the remaining $d-1$ coordinates vary. For example, when $d=2$, the arrays $A_{1;a}$ and  $A_{2;a}$ correspond to the $a$-th row and the $a$-th column of $A$, respectively. When $d=3$, the arrays $A_{1;a}, A_{2;a}, A_{3;a}$ correspond to the  $a$-th layers parallel to the three coordinate axes. 

Let $H$ and $L$ be a Hadamard matrix, and a Latin square of order $n$, respectively, and let $C_\ell$ be a square matrix of order $n$ such that 
\begin{align*}
    C_\ell (i,j) =H(\ell,i)H(\ell,j) \quad \text {for }i,j,\ell\in [n].
\end{align*}
Kharaghani showed that the matrix of order $n^2$ obtained by replacing $\ell$ in $L$ with $C_\ell$ is a Hadamard matrix of order $n^2$ \cite{MR810264}. 
We extend this result to the higher-dimensional setting.

For $n,d,r\in \mathbb{N}$ and $0\leq t\leq d-r$, an {\it $(n,d,r,t)$ Latin hypercube} is a $d$-dimensional hypercube of order  $n$ on $n^r$
symbols such that each $(d-t)$-layer contains each symbol $n^{d-r-t}$ times. So, a Latin square of order $n$ is an $(n,2,1,1)$ Latin hypercube. Some authors refer to $(n,d,1,d-1)$-hypercubes as  Latin hypercubes or permutation hypercubes but there is no consensus on this terminology.
Here is an example of an  $(n,d,1,d-1)$-hypercube $L$  where each line contains each of the $n$ symbols exactly once. 
\begin{align*}
L(i_1,\dots,i_d)=i_1+\dots+i_d \pmod n  & \text { for } i_1,\dots,i_{d}\in [n]. 
\end{align*}

The study of Latin hypercubes was initiated in the 1940's by preeminent statisticians including  Kishen \cite {MR34743}; for some recent results, we refer the reader to \cite{MR2860603,MR3600882,MR4665304}

Before we construct  $(n,d,r,1)$ Latin hypercubes, we need some background. 
A {\it hypergraph} $G$ is a pair $(V,E)$ where $V$ is a finite set called the {\it vertex} set, $E$ is the {\it edge} set, where every edge is a subset of $V$. The {\it degree} of  $v\in V$ in $G$, written $\deg_G(v)$, is the number of occurrences of $v$ in $E$. The {\it complete $d$-uniform $d$-partite} hypergraph with part sizes $n$, written $K_{n\times d}^d$, is  a hypergraph $G=(V,E)$  satisfying the following conditions.
    \begin{align*}
        &\{V_1,\dots, V_d\} \mbox { is a partition of } V \mbox{ where } |V_1|=\dots=|V_d|=n;\\
        &E=\{e\subseteq V \ | \ |e\cap V_i|=1 \mbox{ for } i\in [d]\}.
    \end{align*}
A {\it $k$-edge-coloring}  of $G$ is  a partition of $E$ into {\it color classes} $G(1),\dots,G(k)$; such a coloring is an {\it $s$-factorization} if  $\deg_{G(i)}(v)=s$ for  $v\in V$ and $i\in[k]$. By Baranyai's theorem \cite{MR535941}, $K_{n\times d}^d$ is $s$-factorable if and only if $s \ | \ n^{d-1}$.  
\begin{theorem}
    For $n,d,r\in \mathbb{N}$ and $r\leq d-1$, there exists an $(n,d,r,1)$ Latin hypercube.
\end{theorem}
\begin{proof}
Let $\{V_1,\dots, V_d\}$ be the vertex partition of $G:=K_{n\times d}^d$, and let us assume each $V_1=\dots=V_d=[n]$. Since $n^{d-r-1} \ |\ n^{d-1}$,  by Baranyai's theorem, $G$ is $n^{d-r-1}$-factorable.  So the edges of $G$ can be partitioned into $k$ color classes $G(1),\dots,G(k)$, where $k:=n^{d-1}/n^{d-r-1}=n^r$, and each color class is $n^{d-r-1}$-regular. Now, let $L$ be a $d$-dimensional hypercube of order $n$. For $j\in [n^r]$, we place symbol $j$ in position $(i_1,\dots,i_d)$ of $L$ if there is an edge $e\in G(j)$ such that  $e\cap V_t=\{i_t\}$ for $t\in [d]$. 
Then $L$ is an $(n,d,r,1)$ Latin hypercubes. 
\end{proof}

An $(n,d,r,t)$ Latin hypercube is {\it line-rainbow} if each line contains $n$ distinct symbols.
\begin{remark} \label{trivialremark}\textup{
    If $L$ is an $(n,d,r,t)$ Latin hypercube, then
\begin{enumerate}
    \item  $L$ is an  $(n,d,r,s)$ Latin hypercube for $s\leq t$;
    \item  $L$ is  line-rainbow for $t=d-r$.
\end{enumerate}
}\end{remark} 
\noindent To see why $(n,d,r,d-r)$ Latin hypercubes are line-rainbow, observe that any line is contained in an $r$-layer. But each $r$-layer contains each symbol exactly once. Therefore, each line contains distinct symbols.  

Now, we present a construction of a line-rainbow  $(n,d,r,1)$ Latin hypercubes.  
\begin{theorem} \label{lemmaLrainbowcube}
    For $n,d,r\in\mathbb{N}$, there exists a line-rainbow $(n,d,r,1)$ Latin hypercube.  
\end{theorem}
\begin{proof} For  convenience, we assume the set of symbols is $\{0,\dots, n^r-1\}$.    We define $L$ such that  
    \begin{align*}
        L(i_1,\dots,i_d)= \sum_{j=0}^{r-1} q_j n^{j}, \quad i_1,\dots,i_d \in [n],
    \end{align*}
where $q_j:=\left[i_{d-j} - (i_1+\dots+i_{d-r})\right]_n$ for $0\leq j\leq r-1$.
Recall that a line is obtained by fixing $d-1$ coordinates, while allowing one coordinate, say the $\ell$-th coordinate, vary. To show that $L$ is line-rainbow, there are two cases to consider. 
\begin{itemize}
    \item $\ell\in [d-r]$: Suppose, to the contrary, that  for $a\neq b$,

Then, we have 
$$
[a-b]_n n^{r-1} + [a-b]_n n^{r-2} +\dots+[a-b]_n n +[a-b]_n  =0. 
$$
so $a=b$, which is a contradiction.
    \item $\ell\in [d] \setminus[d-r]$: Suppose, to the contrary, that  for $a\neq b$,
    $$L(i_1,\dots, i_{d-r},\dots, i_{\ell-1}, a, i_{\ell+1},\dots,i_d)=L(i_1,\dots, i_{d-r},\dots, i_{\ell-1}, b, i_{\ell+1},\dots,i_d).$$ 
Then, we have 
 $$ 
 \left[a-(i_1+\dots+i_{d-r})\right]_n n^{d-\ell}=\left[b-(i_1+\dots+i_{d-r})\right]_n n^{d-\ell}, 
 $$
or equivalently, $[a-b]_n n^{d-\ell}=0$, so we have $a=b$, which is a contradiction.
\end{itemize}
Now, we show that $L$ is an $(n,d,r,1)$ Latin hypercube. We need to show that each hyperplane of $L$ contains each symbol equally often.  Let $p$ be a symbol. We can write $p$ as $\sum_{j=0}^{r-1} p_j n^j$ where $0\leq p_j \leq n-1$, $0\leq j\leq r-1$. It suffices to show that for $i_1,\dots,i_d \in [n]$, and fixed $i_\ell$ with $\ell\in [d]$, the following equation has $n^{d-r-1}$ solutions.
\begin{align}\label{layereqeq}
 \sum_{j=0}^{r-1}
q_j n^{j}=\sum_{j=0}^{r-1} p_j n^j.
\end{align}  
Equation \eqref{layereqeq} can be simplified to $\sum_{j=0}^{r-1}
(q_j-p_j) n^{j}=$ which has a solution if and only if $q_j=p_j$ for $0\leq j\leq r-1$, or equivalently, 
\begin{align} \label{congeqsimp}
    i_{d-j} - (i_1+\dots+i_{d-r})=p_j \pmod{n},  \quad 0\leq j\leq r-1.
\end{align}
To complete the proof, there are two cases to consider.
\begin{itemize}
    \item $\ell\in [d-r]$: For $d-r-1$ of the elements in $\{i_1,\dots,i_{d-r}\}$ we have $n$ choices each, and those lead to a unique choice for $i_{d-j}$ for $0\leq j\leq r-1$. Thus, the number of solutions to \eqref{layereqeq} is $n^{d-r-1}$.
    \item  $\ell\in [d]\setminus[d-r]$: In the equation $i_{\ell} - (i_1+\dots+i_{d-r})=p_{d-\ell} \pmod{n}$, for $d-r-1$ of the elements in $\{i_1,\dots,i_{d-r}\}$ we have $n$ choices each, and the remaining element can be chosen uniquely. Consequently, in each of the remaining $r-2$ equations in \eqref{congeqsimp}, $i_{d-j}$ can be uniquely determined for $0\leq j\leq r-1,j\neq d-\ell$. Once again, the number of solutions to \eqref{layereqeq} is $n^{d-r-1}$.
\end{itemize}
\end{proof}

An  {\it $(n,d)$-orthogonal family} is a  collection $\mathcal S$ of $n$ $d$-dimensional $\pm 1$-hypercubes of order $n$    satisfying the following conditions.
\begin{enumerate}
        \item [\textup{(O1)}]  $A_{i;a}\cdot B_{i;b}=0$ for  distinct $A,B\in \mathcal{S}$, $i\in [d]$, and  $a,b\in [n]$, 
        \item [\textup{(O2)}]  $\sum_{A\in \mathcal{S}} A_{i;a}\cdot A_{i;b}=0$ for $i\in [d]$, and distinct $a,b\in [n]$.
\end{enumerate}

The following lemma provides a construction of $(n,d)$-orthogonal families using Hadamard matrices of order $n$. 
\begin{lemma}\label{lem:constA}
  Let  $H$ be a Hadamard matrix of order $n$, and for $\ell\in [n]$, let $C_\ell$ be a $d$-dimensional hypercubes with
    $$C_\ell(i_1,\ldots, i_d)=\prod_{j=1}^d H(\ell,i_j) \quad \text{ for     } i_1,\ldots, i_d\in [n].$$ 
    Then $\mathcal{S}:=\{C_\ell \mid \ell\in [n]\}$ forms an $(n,d)$-orthogonal family. \end{lemma}
\begin{proof}
    Let $i\in [d]$. Without loss of generality, we may assume that $i=1$.
To prove (O1), let  $A,B\in \mathcal{S}$,  and  $a,b\in [n]$. There exist distinct $k,\ell\in [n]$ such that    $A=C_k, B=C_\ell$.  
    We have
    \begin{align*}
        A_{1;a}\cdot B_{1;b}&=\sum_{x_2,\ldots,x_d=1}^n C_k(a, x_2,\ldots, x_d)\, C_{\ell}(b, x_2,\ldots, x_d)\\
        &=\sum_{x_2,\ldots,x_d=1}^n \Bigl(H(k, a)\prod_{j=2}^{d} H(k, x_j)\Bigr) \Bigl(H(\ell, b)\prod_{j=2}^{d} H(\ell, x_j)\Bigr) \\
        &=H(k, a)H(\ell, b)\sum_{x_2,\ldots,x_d=1}^n \prod_{j=2}^{d} H(k, x_j) H(\ell, x_j)\\
        &=H(k, a)H(\ell, b) \prod_{j=2}^{d} \sum_{x=1}^n H(k, x)H(\ell, x)\\
        &=H(k, a)H(\ell, b) (H_{1;k} \cdot H_{1;\ell})^{d-1} =0,
    \end{align*}
    where the last equality follows from the fact that $H$ is a Hadamard matrix. To prove (O2), let $a,b\in [n]$ with $a\neq b$. We have
    \begin{align*}
    \sum_{A\in \mathcal{S}} A_{1;a}\cdot A_{1;b}&=\sum_{\ell=1}^n \sum_{x_2,\ldots,x_d=1}^n C_\ell(a,x_2,\ldots, x_d) C_{\ell}(b,x_2,\ldots, x_d)\\
    &=\sum_{\ell=1}^n \sum_{x_2,\ldots,x_d=1}^n \Big( H(\ell, a)\prod_{j=2}^{d} H(\ell, x_j) \Big)  \Big(H(\ell, b) \prod_{j=2}^{d} H(\ell, x_j)\Big)\displaybreak[0]\\
    &=\sum_{\ell=1}^n \sum_{x_2,\ldots,x_d=1}^n H(\ell, a)H(\ell, b) \prod_{j=2}^{d} H(\ell, x_j)^2\displaybreak[0]\\
    &=\sum_{\ell=1}^n \sum_{x_2,\ldots,x_d=1}^n H(\ell, a)H(\ell, b)\displaybreak[0]\\
    &=n^{d-1}\sum_{\ell=1}^n H(\ell, a)H(\ell, b)=0,
    \end{align*}
     where the last equality follows from the fact that $H$ is a Hadamard matrix. 
\end{proof}
 Recall that in an  $(n,d,r)$-Hadamard hypercube  every two parallel $r$-layers are orthogonal. Lemma~\ref{lem:constA} can be generalized as follows.

\begin{lemma}\label{lem:constA1}
    Let $H$ be an $(n,r+1,r)$-Hadamard hypercube, and for $\ell\in [n]$, let $C_\ell$ be an $m$-dimensional hypercubes where $m:=dr$ and 
$$    
C_\ell(i_1,\ldots, i_{m})=\prod_{j=1}^{d} H(\ell,i_{(j-1)r+1},\ldots, i_{jr}) \quad\text{ for } i_1,\dots,i_{m}\in [n].
$$    
Then  $\mathcal{S}:=\{C_\ell \mid \ell\in [n]\}$ forms an $(n, m)$-orthogonal family.
\end{lemma}
\begin{proof}
    Let $i\in [m]$. Without loss of generality, we may assume that $i=1$.  
To prove (O1), let  $A,B\in \mathcal{S}$,  and  $a,b\in [n]$. There exist distinct $k,\ell\in [n]$ such that    $A=C_k, B=C_\ell$.  
    We have  
\begin{align*}
A_{1;a}\cdot B_{1;b}
&= \sum_{x_2,\ldots,x_m=1}^n 
C_k(a, x_2,\ldots, x_m)\,C_{\ell}(b, x_2,\ldots, x_m) \\
&= \sum_{x_2,\ldots,x_m=1}^n 
\Big(
H(k, a, x_2,\ldots, x_r)\,H(\ell, b, x_2,\ldots, x_r) \\
&\qquad\qquad \times \prod_{j=2}^{d} 
H\big(k, x_{(j-1)r+1},\ldots, x_{jr}\big)\,
H\big(\ell, x_{(j-1)r+1},\ldots, x_{jr}\big)
\Big) \\
&= \left(\sum_{x_2,\ldots,x_r=1}^n
H(k,a,x_2,\ldots,x_r)\,
H(\ell,b,x_2,\ldots,x_r)\right) \\
&\quad \times \prod_{j=2}^{d}
\left(\sum_{x_{(j-1)r+1},\ldots,x_{jr}=1}^n
H\big(k,x_{(j-1)r+1},\ldots,x_{jr}\big)\,
H\big(\ell,x_{(j-1)r+1},\ldots,x_{jr}\big)\right) \\
&= 0.
\end{align*}    
    where the last equality follows from the fact that $H$ is an $(n,r+1,r)$ Hadamard hypercube. To prove (O2), let $a,b\in [n]$ with $a\neq b$. We have
\begin{align*}
\sum_{A\in \mathcal{S}} A_{1;a}\cdot A_{1;b}
&= \sum_{\ell=1}^n \sum_{x_2,\ldots,x_m=1}^n
C_\ell(a, x_2,\ldots, x_m)\,C_\ell(b, x_2,\ldots, x_m) \\
&= \sum_{\ell,x_2,\ldots,x_m=1}^n
H(\ell, a,x_2,\ldots, x_r)\,
H(\ell, b,x_2,\ldots, x_r) \\
&\quad \times \prod_{j=2}^{d}
H\big(\ell, x_{(j-1)r+1},\ldots, x_{jr}\big)^2 \\
&= \sum_{\ell,x_2,\ldots,x_m=1}^n
H(\ell, a,x_2,\ldots, x_r)\,
H(\ell, b,x_2,\ldots, x_r) \\
&= n^{(d-1)r}
\sum_{\ell,x_2,\ldots,x_r=1}^n
H(\ell, a,x_2,\ldots, x_r)\,
H(\ell, b,x_2,\ldots, x_r) = 0.
\end{align*}  
     where the last equality follows from the fact that $H$ is an $(n,r+1,r)$ Hadamard hypercube. 
\end{proof}
Using line-rainbow   $(n,d,r,1)$ Latin hypercubes and $(n^r,d)$-orthogonal  families, one can construct $(n^{r+1},d,d-1)$-Hadamard hypercubes.
\begin{lemma}\label{lem:consthcube}
    Let $L$ be a line-rainbow $(n,d,r,1)$ Latin hypercube, and $\mathcal S:=\{C_1,\dots,C_{n^r}\}$ be an $(n^r,d)$-orthogonal family.    
 The $d$-dimensional hypercube $H$ of order $n^{r+1}$ obtained by replacing each $\ell\in [n^r]$ in $L$ with $C_\ell$ is an $(n^{r+1},d,d-1)$-Hadamard hypercube. 
\end{lemma}
\begin{proof}
    We need to show that any two parallel hyperplanes $K$ and $K'$ in $H$ are orthogonal. By an appropriate permutation of the coordinates, we can assume—without loss of generality—that $K$ and $K'$ are formed by fixing the first coordinate of $H$ to $a$ and $a'$, respectively.  We can write $a,a'\in [n^{r+1}]$ as $a = (a_1-1)n + a_2$ and $a' = (a_1'-1)n + a_2'$, where $a_1, a_1' \in [n^r]$ and $a_2, a_2' \in [n]$.  
  We have  
    \begin{align*}
    K \cdot K'=\sum_{{\vb x}\in[n]^{d-1}} (C_{L(a_1, \vb{x})})_{1;a_2}\cdot  (C_{L(a'_1, \vb{x})})_{1;a'_2}. 
    \end{align*}
  If $a_1=a'_1=:b$, then have 
    \begin{align*}
        K \cdot K'=\sum_{\vb{x}\in[n]^{d-1}} (C_{L(b,\vb{x})})_{1;a_2}\cdot  (C_{L(b,\vb{x})})_{1;a'_2}=n^{d-r-1}\sum_{\ell=1}^{n^r}  (C_{\ell})_{1;a_2}\cdot (C_{\ell})_{1;a'_2}=0, 
    \end{align*}
    where the first equality follows from the fact that $L$ is an $(n,d,r,1)$ Latin hypercube, and the second equality follows from (O2). If $a_1\neq a'_1$,  then each term of $K\cdot K'$ is zero by (O1) and the fact that $L$ is line-rainbow, and thus $K\cdot K'=0$.  
\end{proof}

We are ready to prove the main result of this section. 
\begin{theorem}\label{thm:hypercube} 
       Let $n, r, d \in \mathbb{N}$ with $d \geq 2$. If there exists an $(n^r, s+1, s)$-Hadamard hypercube, then there also exists an $(n^{r+1}, ds, ds-1)$-Hadamard hypercube. 
\end{theorem}
\begin{proof}
Let $H$ be an  $(n^r,s+1,s)$-Hadamard hypercube. By Lemma \ref{lem:constA1}, there exists an $(n^r, ds)$-orthogonal family $\mathcal S$. By Theorem \ref{lemmaLrainbowcube}, there exists  a line-rainbow $(n,d,r,1)$ Latin hypercube $L$. Applying Lemma \ref{lem:consthcube} completes the proof. 
\end{proof}

In particular, if there exists a Hadamard matrix of order $n^r$, then for any $d\geq 2$, there exits an $(n^{r+1},d,d-1)$-Hadamard hypercube.

\section{High-Dimensional Symmetric Designs}\label{subsec:hcube1}

A Hadamard matrix is  \textit{normalized} if all entries in its first row and first column are equal to $+1$. 
Recall that a symmetric $(n,d, k,\lambda)$-design is a  $\{0,1\}$-hypercube of order $n$ and dimension $d$ with the property that every 
 $2$-layer is the incidence matrix of a symmetric $(n,k,\lambda)$-design.

We shall skip the proof of the following lemma. 
\begin{lemma}\label{lem:cmat}
Let $H$ be a Hadamard matrix of order $n$, and let $C_\ell$ be a square matrix of order $n$ such that $C_\ell (i,j) =H(\ell,i)H(\ell,j)$ for $i,j,\ell\in [n]$. We have the following.
    \begin{enumerate}
        \item [\textup{(i)}]  For $\ell\in [n]$, $C_\ell^2=n C_\ell$; 
        \item [\textup{(ii)}] \label{item:2} For distinct $\ell,\ell'\in [n]$, $C_\ell C_{\ell'}=O$;
        \item [\textup{(iii)}]  $\sum_{\ell=1}^n C_\ell=nI$. 
    \end{enumerate}
\end{lemma}

A $\pm 1$-array $A$ of order $n$ is {\it $1$-compatible} with the Latin square $L$ of order $n$, if $A(i,j)=1$ whenever $L(i,j)=1$ for $i,j\in [n]$. 

\begin{lemma}\label{lem:hm} 
Let $H$ be a Hadamard matrix, $L$ a Latin square,  and $A$ a $\pm 1$-array, all of order $n$. Let $C_1,\dots,C_n$ be $\pm 1$-arrays of order $n$ such that $C_\ell (i,j) =H(\ell,i)H(\ell,j)$ for $i,j,\ell\in [n]$. 
Let $M$ be an $n \times n$ block matrix whose blocks are themselves $n \times n$ matrices, where the $(i,j)$-block is given by
\begin{align*}
M_{ij} = A(i,j)\, C_{L(i,j)} \quad \text{for } i,j \in [n].
\end{align*}
Then $M$ is a Hadamard matrix. Moreover, if $H$ is normalized and $A$ is $1$-compatible with $L$, then $(J-M)/2$ is the incidence matrix of a symmetric $(n^2,n(n-1)/2,n(n-2)/4)$-design. 
\end{lemma}
\begin{proof}
    It is clear that $M$ is a   $\pm 1$-array of order $n^2$. Using Lemma \ref{lem:cmat}, the $(i,j)$-block of $MM^\top$ is given by
\begin{align*}
\sum_{k=1}^n M_{ik}M^\top_{kj}
&= \sum_{k=1}^n A(i,k)\,C_{L(i,k)} \, A(j,k)\,C_{L(j,k)} \\
&= \sum_{k=1}^n A(i,k)A(j,k)\, C_{L(i,k)}C_{L(j,k)} \\
&= n \sum_{k=1}^n A(i,k)A(j,k)\, \delta_{L(i,k),L(j,k)}\, C_{L(i,k)} \\
&= \begin{cases}
\displaystyle n \sum_{k=1}^n  C_{L(i,k)} 
= n^2 I & \text{if } i=j,\\[6pt]
0 & \text{if } i \neq j.
\end{cases}
\end{align*}    
    Therefore, $MM^\top=n^2I$, that is, $M$ is a Hadamard matrix of order $n^2$.   
    
 Now, suppose that  $H$ is normalized and $A$ is $1$-compatible with $L$. We have  $C_1=J_n$, and so we have $C_1 J_n=J_n C_1=n J_n$ and by Lemma~\ref{lem:cmat}(ii),  $C_i J_n=C_i C_1 = 0= C_1 C_i=J_n C_i$ for $i\geq 2$. 
   The $(i,j)$-block of $MJ$ is given by 
   \begin{align*}
(MJ)_{ij}
&= \sum_{k=1}^n M_{ik}J_{kj} \\
&= \sum_{k=1}^n A(i,k)\,C_{L(i,k)}\,J_n \\
&= \sum_{k=1}^n A(i,k)\,\big(C_{L(i,k)}J_n\big) \\
&= \sum_{k:\,L(i,k)=1} A(i,k)\,nJ_n \\
&= nJ_n.
\end{align*}
Thus, $MJ=nJ_{n^2}$, and so $JM^\top= (MJ)^\top= (nJ_{n^2})^\top= nJ_{n^2}$.
Finally, we show that $K:=(J-M)/2$ is the incidence matrix of a symmetric $(n^2,n(n-1)/2,n(n-2)/4)$-design.  To see this, first observe that $2KJ=J^2-MJ= n^2 J-nJ=n(n-1) J$ and similarly, $2JK=n(n-1)  J$. Moreover, $4 KK^\top= (J-M)(J-M^\top) =J^2 - JM^\top - MJ + MM^\top = n^2J - nJ - nJ + n^2I = n^2 I + n(n-2)J$.    
\end{proof}

For the purpose of next lemma, we need to introduce some notation.  Recall that if $A$ is a $d$-dimensional hypercube of order $n$, then for any $i \in [d]$ and $a \in [n]$,  $A_{i;a}$ denotes the $(d-1)$-dimensional layer obtained by fixing the $i$-th coordinate to $a$ and allowing the remaining $d-1$ coordinates to vary. More generally, for a subset $I \subseteq [d]$ and a vector $\mathbf{a} \in [n]^{|I|}$, let $A_{I;\mathbf{a}}$ denote the $(d - |I|)$-dimensional layer obtained by fixing the coordinates indexed by $I$ according to $\mathbf{a}$, while the remaining coordinates vary. For every $x\in [n^2]$, let $\overline x, \rttensor x$ be the unique integers such that  $x=(\overline x-1) n + \rttensor x$ and $\overline x, \rttensor x\in [n]$.

 We are ready to prove the main result of this section.
\begin{theorem}\label{thm:hcubeproper}
    If there exists a Hadamard matrix of order $n$, then there exists an $(n^2,d,1)$-Hadamard hypercube, and a symmetric $(n^2,d, n(n-1)/2,n(n-2)/4)$-design. 
\end{theorem}
\begin{proof}
 Let $H$ be a Hadamard matrix of order $n$, and let $R$ be an $(n,d,1,d-1)$ Latin hypercube. For $\ell\in [n]$, we define the $d$-dimensional hypercube $D_\ell$ of order $n$ by $D_\ell(x_1,\ldots,x_d)=\prod_{t=1}^d H(\ell,x_t)$ for $x_1,\ldots, x_d\in [n]$. 
  We construct the hypercube $N$ of order $n^2$   by replacing each entry $\ell\in [n]$ in $R$ with $D_\ell$. For $\vb x:=(x_1,\dots,x_d)\in [n^2]^d$, and $\overline{\vb x}:=(\overline x_1,\dots, \overline x_d), \rttensor{\vb x}:=(\rttensor x_1,\dots,\rttensor x_d)$, we have
$$
N(\vb x)=D_{R(\overline{\vb x})}(\rttensor{\vb x}).
$$
Now, let us fix $I:=[d-2]$ and $\vb a:=(a_1,\dots, a_{d-2})\in [n^2]^{d-2}$. Let $M=N_{I;\vb a}$,  $L=R_{I;\overline{\vb a}}$, and we define $A$ such that
$$A(i,j)=\prod_{t=1}^{d-2} H\Big(L(i,j), \rttensor a_i\Big), \quad i,j\in [n].$$
It is clear that $A$ is a $\pm 1$-array and that if $H$ is normalized, then $A$ is 1-compatible with $L$. For $u,v\in [n^2]$, we have 
\begin{align*}
    M(u,v)&=N(a_1,\dots,a_{d-2},u,v)=D_{R(\overline{a_1},\dots,\overline{a_{d-2}},\overline{ u}, \overline{v})}(\rttensor a_1,\dots, \rttensor a_{d-2}, \rttensor u, \rttensor v)
   = A(\overline{ u}, \overline{v}) C_{L(\overline {u},\overline{v})}(\rttensor u, \rttensor v).
\end{align*}
By Lemma \ref{lem:hm}, $M$ is a Hadamard matrix, and that $(J-M)/2$ is the incidence matrix of a symmetric $(n^2,n(n-1)/2,n(n-2)/4)$-design whenever  $H$ is normalized. The entire argument is valid if we replace $I$ by any  $(d-2)$-subset of $[d]$, and so  $N$ is an $(n^2,d,1)$-Hadamard cube, and  $(J-N)/2$ is a symmetric $(n^2,d, n(n-1)/2,n(n-2)/4)$-design whenever $H$ is normalized.     
\end{proof}

\subsubsection*{Acknowledgments}
Amin Bahmanian's research is partially supported by a  Faculty Research Award at Illinois State University. Sho Suda's research is supported by JSPS KAKENHI Grant Number 22K03410, 26K06904.

\bibliographystyle{plain}
\bibliography{bibHad}

\end{document}